\documentclass[12pt]{article}
\usepackage{fullpage,enumerate}
\usepackage[utf8]{inputenc}
\usepackage[intlimits]{amsmath}
\usepackage{amssymb,latexsym,amsthm,mathtools,tikz}
\usepackage{hyperref}
\usepackage[textmode]{ytableau}
\usepackage[font=footnotesize,labelfont=bf]{caption}

\allowdisplaybreaks[3]

\newcommand{\since}[1]{\tag*{\scriptsize{(#1)}}}
\newcommand{\smallfrac}[2]{\textstyle\frac{#1}{#2}\displaystyle}
\newcommand{\Perm}{S}
\newcommand{\Cyc}{C}
\newcommand{\alphacyc}{\alpha^{\text{cyc}}}
\newcommand{\betacyc}{\beta^{\text{cyc}}}

\DeclareMathOperator{\mob}{M\ddot{o}b}
\DeclareMathOperator{\lcm}{lcm}
\DeclareMathOperator{\type}{type}
\DeclareMathOperator{\per}{per}
\DeclareMathOperator{\ev}{ev}
\DeclareMathOperator{\co}{co}
\newcommand{\PP}{\mathbb{P}}
\DeclareMathOperator{\Alt}{Alt}
\DeclareMathOperator{\alt}{alt}
\newcommand{\gam}{\gamma}
\newcommand{\ee}{\chi}
\newcommand{\ZZ}{\mathbb{Z}}
\newcommand{\card}[1]{{\lvert #1 \rvert}} 	% cardinality

\newtheorem{thm}{Theorem}[section]
\newtheorem{prop}[thm]{Proposition}
\newtheorem{lem}[thm]{Lemma}
\newtheorem{cor}[thm]{Corollary}
\newtheorem{conj}[thm]{Conjecture}
\newtheorem*{question}{Question}
\theoremstyle{definition}
\newtheorem{defn}[thm]{Definition}

\everymath{\displaystyle}
\renewcommand{\qedsymbol}{$\blacksquare$}

\begin{document}

\title{Exact and asymptotic enumeration of cyclic permutations according to descent set}
\date{March 14, 2019}
\author{Sergi Elizalde\thanks{Department of Mathematics,
Dartmouth College, Hanover, NH 03755. E-mail: \url{sergi.elizalde@dartmouth.edu}. Partially supported by Simons Foundation grant \#280575.}
\and Justin M. Troyka\thanks{Department of Mathematics,
Dartmouth College, Hanover, NH 03755. E-mail: \url{jmtroyka@gmail.com}.}}
\maketitle

\begin{abstract}
Using a result of Gessel and Reutenauer, we find a simple formula for the number of cyclic permutations with a given descent set, by expressing it in terms of ordinary descent numbers (i.e., those counting all permutations with a given descent set). We then use this formula to show that, for almost all sets $I \subseteq [n-1]$, the fraction of size-$n$ permutations with descent set $I$ which are $n$-cycles is asymptotically $1/n$. As a special case, we recover a result of Stanley for alternating cycles. We also use our formula to count $n$-cycles with no two consecutive descents.
\end{abstract}

\section{Introduction} \label{Introduction}

Two of the most common ways to represent a permutation are its one-line notation, which views the permutation as a word,
and its expression as a product of disjoint cycles, which depicts the algebraic nature of the permutation.
Many known properties of permutations concern one of these two perspectives.
For example, by considering the one-line notation of a permutation, one can study descent sets, pattern avoidance, longest increasing subsequences, etc.
On the other hand, by viewing a permutation as a product of cycles, one can study its number of fixed points, its number of cycles, whether it is an involution, etc.

However, the interplay between these two depictions of permutations is far from being understood. 
We are particularly interested in how the cycle structure of a permutation (with a focus on permutations that consist of one cycle) and its descent set relate.
The first major breakthrough in this area is the seminal paper of Gessel and Reutenauer~\cite{GR}, which expresses the number of permutations with a given cycle structure and descent set as an inner product of symmetric functions,
in the same vein as previous formulas by Gessel~\cite{Gessel} counting permutations with a given descent set and inverse descent set.
Some of the results in~\cite{GR} relate to the work of Diaconis, McGrath, and Pitman~\cite{DMP}, who approach problems involving descents and cycles from the perspective of probability and riffle shuffles. For the special case of cyclic permutations (sometimes called {\em cycles} or {\em $n$-cycles}), an unexpected property of the distribution of descent sets was later given in~\cite{Elizalde}, showing that descent sets of $n$-cycles, when restricted to the first $n-1$ entries, have the same distribution as descent sets of permutations of length $n-1$. Around the same time, using results from~\cite{GR}, Stanley gave a formula for the number of cycles which are alternating~\cite{StanleySurvey}. He used it to show that, on a random permutation, the events ``being a cycle" and ``being alternating" are independent, in a precise sense that will be discussed later.

Other work has explored the framework of Gessel and Reutenauer~\cite{GR} in more general settings. Poirier~\cite{PoirierThesis,Poirier} replaces the group of permutations with a certain wreath product, whose elements can be thought of as ``$r$-colored'' permutations for given $r$. Their results on counting elements of the wreath product with given descent set and cycle type are analogous to those in~\cite{GR}, and they specialize to the case of the hyperoctahedral group (the group of ``two-colored'' permutations). Some of these results on the hyperoctahedral group were found previously by Reiner~\cite{Reiner}. More recently, Novelli, Reutenauer, and Thibon \cite{NRT} generalize the work of Gessel and Reutenauer in a different way: instead of looking at the consecutive pairs of values to define the descent set of a permutation, they look at the consecutive $k$-tuples of values to define the ``$k$-descent code'' of a permutation. In the combinatorics of Hopf algebras, descent sets of permutations have a deep connection to non-commutative symmetric functions and quasi-symmetric functions, and \cite{NRT} defines Hopf algebras that have the same connection to $k$-descent codes.

After introducing some notation and background in Section~\ref{NotationSection}, we present our main theorem in Section~\ref{MainTheoremSection}. 
It gives four equivalent formulas relating the enumeration of cyclic permutations with a given descent set and the enumeration of all permutations with a given descent set, the latter being a well-understood problem. In Section~\ref{sec:consequences} we discuss some consequences of our main theorem, including the result from~\cite{Elizalde} mentioned above, and Stanley's formula counting alternating cycles, as well as some extensions of it. In Section~\ref{AsymptoticSection} we study the asymptotic implications of our main theorem, generalizing Stanley's observation about the independence of being cyclic and being alternating, and conjecturing that it can be even further generalized.

The last section of the paper relates to pattern avoidance. The first significant result about the interaction of pattern avoidance and cycle structure of permutations is due to Robertson, Saracino and Zeilberger~\cite{RSZ},  who showed that the number of fixed points has the same distribution on permutations avoiding two different patterns of length~3. This was later generalized in~\cite{Elifpexc,EliPak} by including also the distribution of the number of excedances. 
At the conference Permutation Patterns 2010, Stanley proposed the problem of enumerating cyclic permutations that avoid a given pattern of length $3$. While this problem remains unsolved, cyclic permutations avoiding some specific sets of patterns ($\{321,2143,3142\}$ and $\{213,312\}$)
have been enumerated by Archer and Elizalde~\cite{ArcEli}. In addition, permutations with restrictions that involve both their one-line notation and their cycle structure (e.g. cycles whose excedances are increasing) are counted in~\cite{Elicont} using continued fractions. 

In Section~\ref{sec:monotone} we consider a related problem, namely that of counting cycles that avoid a monotone {\em consecutive} pattern.  Avoiding a monotone consecutive pattern of length $k$ is equivalent to not having $k-1$ consecutive ascents (or descents).
In general, consecutive patterns differ from classical patterns by adding the requirement that the entries in an occurrence of the pattern have to be adjacent; see~\cite{Eli_survey} for a survey of the literature in the subject. In Section~\ref{MonotoneSubsection1}, the number of permutations with a given cycle type avoiding a consecutive pattern of length $k$ is expressed as an inner product of symmetric functions. In Section~\ref{cyc123}, we use this, along with our main theorem, to give an explicit formula counting permutations avoiding the consecutive pattern $\underline{123}$ or $\underline{321}$.

\section{Notation and definitions} \label{NotationSection}

In this section we set up the notation that is used through the rest of the paper.

Let $n\ge1$, and let $\Perm_n$ denote the symmetric group on $[n]=\{1,2,\dots,n\}$. For a permutation  $\pi \in \Perm_n$, let $\type(\pi)$ denote the cycle type of $\pi$, that is, the partition of $n$ whose parts are the lengths of the cycles of $\pi$.
Let $D(\pi) \subseteq [n-1]$ denote the descent set of $\pi$, that is,
\[ D(\pi) = \{i \colon \pi(i) > \pi(i+1)\}. \]
Let $\Cyc_n$ denote the set of permutations in $\Perm_n$ whose cycle type is $(n)$. We call the elements of $\Cyc_n$ {\em cyclic permutations}, {\em $n$-cycles}, or simply {\em cycles}. 
\begin{defn} For $I \subseteq [n-1]$, let
\[ \begin{matrix}
\alpha_n(I) = \#\{\pi \in \Perm_n \colon D(\pi) \subseteq I\}; \quad & \alphacyc_n(I) = \#\{\pi \in \Cyc_n \colon D(\pi) \subseteq I\}; \\
\beta_n(I) = \#\{\pi \in \Perm_n \colon D(\pi) = I\}; \quad & \betacyc_n(I) = \#\{\pi \in \Cyc_n \colon D(\pi) = I\}.
\end{matrix}\]
\end{defn}
Thus we have four descent functions $\alpha_n, \beta_n, \alphacyc_n, \betacyc_n \colon 2^{[n-1]} \to \mathbb{N}$. The numbers $\alpha_n(I)$ and $\beta_n(I)$ are well-understood and easy to compute; see \cite[Sec.\ 1.4 \& Example 2.2.4]{Stanley1}.

We use $\mob$ to denote the M\"obius function from number theory. As a reminder, it is defined on positive integers as follows:
\[ \mob(d) = \begin{cases} 1 & \text{if $d$ is square-free and has an even number of prime factors}; \\
-1 & \text{if $d$ is square-free and has an odd number of prime factors}; \\
0 & \text{if $d$ is not square-free}.
\end{cases} \]

For $I$ a set of integers and $n\ge1$, define $(I,n) = \text{gcd}(I \cup \{n\})$; and for integer $d \ge 1$, define $I/d = \{i/d \colon \text{$i \in I$ and $d\mid i$}\}$. The following lemma is straightforward.

\begin{lem} \label{basiclemma} (a) For any integers $d$ and $e$, $(I/e)/d = I/(de)$. (b) If $(I,n) = m$ and $e \mid m$, then $(I/e, n/e) = m/e$. \hfill $\square$
\end{lem}

Given a subset $I \subseteq [n-1]$, let $\co(I)$ denote the associated composition of $n$; that is, if $I = \{i_1 < i_2 < \cdots < i_{k-1}\}$, then $\co(I) = (i_1, i_2 - i_1, i_3 - i_2, \ldots, i_{k-1} - i_{k-2}, n - i_{k-1})$. If $\mu = (\mu_1, \ldots, \mu_k)$ is a composition and $d \mid \mu_j$ for each $j$, then define $\mu/d = (\mu_1/d, \ldots, \mu_k/d)$.

\begin{lem} \label{sillylemma} Let $I \subseteq [n-1]$ with $\co(I) = (\mu_1, \ldots, \mu_k)$, and let $d \mid (I,n)$. Then $d \mid \mu_j$ for each $j$, and $\co(I/d) = \mu/d$. \hfill $\square$
\end{lem}

We use $\lambda\vdash n$ to denote that $\lambda$ is a partition of $n$, and $\mu\vDash n$ to denote that $\mu$ is a composition of $n$. We let $\ell(\mu)$ denote the length of $\mu$, \textit{i.e.}\ the number of parts.

\begin{defn} \label{lyndon} Let $w$ be a (finite) non-empty word on alphabet $\{1, 2, \ldots\}$. We say $w$ is a \textit{primitive word} if $w$ is not equal to any of its non-trivial cyclic shifts: that is, if $w = uv$ with $u$ and $v$ non-empty then $w \not= vu$. We say $w$ is a \textit{Lyndon word} if $w$ is strictly less than all of its non-trivial cyclic shifts in the lexicographic order: that is, if $w = uv$ with $u$ and $v$ non-empty, then $w < vu$ in the lexicographic order.
\end{defn}

Note that every Lyndon word is primitive, and every primitive word has exactly one cyclic shift that is a Lyndon word. It is well-known (see \cite[Thm.\ 5.1.5]{L} or \cite[Exer.\ 7.89.d]{Stanley2}) that every word has a unique factorization into a weakly decreasing (in the lexicographic order) sequence of Lyndon words: that is, for every word $w$ there is a unique sequence of Lyndon words $u_1 \ge u_2 \ge \cdots \ge u_k$ such that $w = u_1u_2\dots u_k$.

\begin{defn} \label{type} Given a word $w$ whose factorization into Lyndon words is $w = u_1u_2\dots u_k$, the \textit{type} of $w$, denoted $\type(w)$, is the partition whose parts are equal to the lengths of the Lyndon words $u_1, u_2, \ldots, u_k$, arranged in order of weakly decreasing length. The \textit{evaluation} of $w$ is the weak composition $\ev(w) = (\mu_1, \mu_2,\ldots)$ such that $\mu_j$ is the number of $j$'s in $w$. The \textit{period} of $w$, denoted $\per(w)$, is the length of the shortest word $v$ such that $w = v^r$ for some $r$.
\end{defn}

Let $x^\mu$ denote the monomial $x_1^{\mu_1}x_2^{\mu_2}\cdots$. Given a partition $\lambda$, let $L_\lambda$ denote the symmetric function
\begin{equation}\label{eq:Ldef} L_\lambda = \sum_{\type(w) = \lambda} x^{\ev(w)}, \end{equation}
where the sum is over all words $w$ with type $\lambda$. We will write $L_n$ instead of $L_{(n)}$. It is well-known (see for instance \cite[eqn.\ (2.2)]{GR}, \cite[Exer.\ 7.89.a]{Stanley2} or \cite[Thm.\ 7.2]{Reutenauer}) that
\[ L_n = \frac{1}{n} \sum_{\text{$d \mid n$}} \mob(d)\,p_d^{n/d}, \]
where $p_d = x_1^d + x_2^d + \cdots$, the power-sum symmetric function.

Given a partition $\lambda$ and a weak composition $\mu = (\mu_1, \mu_2, \ldots)$, let $a_{\lambda,\mu}$ denote the number of words of type $\lambda$ and evaluation $\mu$. We will write $a_{n,\mu}$ instead of $a_{(n),\mu}$. By definition,
\begin{equation}\label{eq:L} 
L_\lambda = \sum_\mu a_{\lambda,\mu} x^\mu,
\end{equation}
where the sum is over all weak compositions $\mu$. 

Recall that the multinomial coefficient $\binom{n}{\mu}$ denotes the total number of length-$n$ words with evaluation $\mu$. The following result is well-known.

\begin{lem}[{\cite[Prop.\ 1.4.1]{Stanley1}}] \label{lem:multi}
$$\alpha_n(I) = \binom{n}{\mu}.$$
\end{lem}

The proof of our main theorem, in Section \ref{MainTheoremSection}, uses the following important result of Gessel and Reutenauer \cite{GR}.

\begin{thm}[{\cite[Cor.\ 2.2]{GR}}] \label{GRcor}
Let $\mu = \co(I)$. The number of $\pi \in \Perm_n$ with $\type(\pi) = \lambda$ and $D(\pi) \subseteq I$ is equal to $a_{\lambda,\mu}$.
%, which is also the coefficient of $x^\mu$ in $L_\lambda$.
In particular, $\alphacyc_n(I) = a_{n,\mu}$, the number of Lyndon words with evaluation $\co(I)$.
%, which is also the coefficient of $x^\mu$ in $L_n$.
\end{thm}

By Equation~\eqref{eq:L}, $a_{n,\mu}$ is the coefficient of $x^\mu$ in $L_n$.
Gessel and Reutenauer \cite{GR} state their results not in terms of words and Lyndon words, but in terms of primitive necklaces and multisets of primitive necklaces; other literature (such as \cite{Reiner}, \cite{Steinhardt}) has used the term ``ornament'' in place of ``multiset of necklaces''. Multisets of primitive necklaces are interchangeable with words for our purposes, because there is a simple bijection between them that preserves type and evaluation: given a multiset of primitive necklaces, cyclically shift each necklace to make it a Lyndon word, then concatenate the resulting Lyndon words in weakly decreasing order to form a word. This is a bijection because of the unique factorization of a word into Lyndon words. In this paper, words are used instead of multisets of primitive necklaces because it makes the main theorem's proof easier.

\section{The main theorem} \label{MainTheoremSection}

Our main result is a relation between the number of permutations with a given descent set and the number of cycles with a given descent set, expressed in four equivalent identities.

\begin{thm} \label{maintheorem}
Let $I \subseteq [n-1]$, and recall the notation $(I,n) = \gcd(I \cup \{n\})$ and $I/d = \{i/d \colon \text{$i \in I$ and $d \mid i$}\}$. We have
\begin{itemize}
\item[(a)] $\alpha_n(I) = \sum_{\text{$d\mid(I,n)$}} \frac{n}{d}\, \alphacyc_{n/d}(I/d)$;
\item[(b)] $\alphacyc_n(I) = \frac{1}{n}\sum_{\text{$d\mid(I,n)$}} \mob(d)\,\alpha_{n/d}(I/d)$;
\item[(c)] $\betacyc_n(I) = \frac{1}{n} \sum_{\text{$d\mid n$}} \mob(d)\,(-1)^{|I|-|I/d|}\beta_{n/d}(I/d)$;
\item[(d)] $\beta_n(I) = \sum_{\text{$d\mid n$}} (-1)^{|I|-|I/d|}\,\frac{n}{d}\,\betacyc_{n/d}(I/d)$.
\end{itemize}
\end{thm}

Before proving this theorem, we will show that the four formulas are equivalent.

\begin{thm} \label{inversions} If any one of the four equations in Theorem \ref{maintheorem} is true for all $n\ge1$ and $I \subseteq [n-1]$, then they are all true.
\end{thm}

\begin{proof} \textbf{(a) $\Rightarrow$ (b): } Fix $n$ and $I$, and set $m = (I,n)$. For $c \mid m$, define
\[ f(c) = \alpha_{cn/m}(I/(m/c)) \qquad \text{and} \qquad g(c) = \frac{cn}{m}\,\alphacyc_{cn/m}(I/(m/c)). \]
Then
\begin{align*} f(c) &= \sum_{\text{$d \mid (I/(m/c), cn/m)$}} \frac{cn}{dm}\,\alphacyc_{\frac{cn}{dm}}((I/(m/c))/d) \since{by (a)} \\
&= \sum_{\text{$d \mid c$}} \frac{cn}{dm}\,\alphacyc_{\frac{cn}{dm}}(I/(dm/c)) \since{by Lemma \ref{basiclemma}} \\
&= \sum_{\text{$d \mid c$}} \frac{dn}{m}\,\alphacyc_{dn/m}(I/(m/d)) \since{substituting $d$ for $c/d$} \\
&= \sum_{\text{$d \mid c$}} g(d).
\end{align*}
Therefore, by M\"obius inversion, $g(c) = \sum_{\text{$d \mid c$}} \mob(d)\,f(c/d)$. Setting $c = m$ yields
\[ n\,\alphacyc_n(I) = \sum_{\text{$d \mid m$}} \mob(d)\,\alpha_{n/d}(I/d), \]
from which (b) follows.

\smallskip

\noindent\textbf{(b) $\Rightarrow$ (c): }
\begin{align*}
\betacyc_n(I) &= \sum_{J \subseteq I} (-1)^{|I|-|J|} \alphacyc_n(J) \since{by Principle of Inclusion--Exclusion} \\
&= \frac{1}{n}\sum_{J \subseteq I} (-1)^{|I|-|J|} \sum_{\text{$d \mid (J,n)$}} \mob(d)\,\alpha_{n/d}(J/d) \since{by (b)} \\
&= \frac{1}{n}\sum_{\text{$d \mid n$}} \mob(d)\sum_{J \subseteq I \cap d\mathbb{Z}} (-1)^{|I|-|J|}\alpha_{n/d}(J/d) \\
&= \frac{1}{n}\sum_{\text{$d \mid n$}} \mob(d)\sum_{J \subseteq I \cap d\mathbb{Z}} (-1)^{|I|-|J|}\sum_{K \subseteq J} \beta_{n/d}(K/d) \\
&= \frac{1}{n}\sum_{\text{$d \mid n$}} \mob(d)\sum_{K \subseteq I \cap d\mathbb{Z}} \beta_{n/d}(K/d) \sum_{K \subseteq J \subseteq I \cap d\mathbb{Z}} (-1)^{|I|-|J|}.
\end{align*}
By the Principle of Inclusion--Exclusion, the sum $\sum_{K \subseteq J \subseteq I \cap d\mathbb{Z}} (-1)^{|I|-|J|}$ (sum over $J$, with $K$ and $I$ fixed) is zero unless $K = I \cap d\mathbb{Z}$, in which case it equals $(-1)^{|I|-|I \cap d\mathbb{Z}|}$. Thus, we obtain
\[ \betacyc_n(I) =\frac{1}{n}\sum_{\text{$d \mid n$}} \mob(d)\,\beta_{n/d}((I\cap d\mathbb{Z})/d)\,(-1)^{|I|-|I \cap d\mathbb{Z}|}. \]
Using that $(I\cap d\mathbb{Z})/d = I/d$ and $|I\cap d\mathbb{Z}| = |I/d|$, we arrive at (c).

\smallskip

\noindent\textbf{(c) $\Rightarrow$ (d): } Fix $n$ and $I$. For $c \mid n$, define
\[ f(c) = (-1)^{|I/(n/c)|} \beta_c(I/(n/c)) \qquad \text{and} \qquad g(c) = (-1)^{|I/(n/c)|} c\,\betacyc_c(I/(n/c)). \]
Then
\begin{align*} g(c) &= (-1)^{|I/(n/c)|}\sum_{\text{$d \mid c$}} \mob(d)\,(-1)^{|I/(n/c)| - |(I/(n/c))/d|} \beta_{c/d}((I/(n/c))/d) \since{by (c)} \\
&= \sum_{\text{$d \mid c$}} \mob(d)\,(-1)^{|I/(dn/c)|} \beta_{c/d}(I/(dn/c)) \since{by Lemma \ref{basiclemma}(a)} \\
&= \sum_{\text{$d \mid c$}} \mob(d)\,f(c/d).
\end{align*}
Therefore, by M\"obius inversion, $f(c) = \sum_{\text{$d \mid c$}} g(d) = \sum_{\text{$d \mid c$}} g(c/d)$. Setting $c = n$ yields
\[ (-1)^{|I|}\,\beta_n(I) = \sum_{\text{$d \mid n$}} (-1)^{|I/d|}\,\frac{n}{d}\,\betacyc_{n/d}(I/d), \]
from which (d) follows.

\smallskip

\noindent\textbf{(d) $\Rightarrow$ (a): }
\begin{align*}
\alpha_n(I) &= \sum_{J \subseteq I} \beta_n(J) = \sum_{J \subseteq I} \sum_{\text{$d\mid n$}} (-1)^{|J|-|J/d|}\,\frac{n}{d}\,\betacyc_{n/d}(J/d) \since{by (d)} \\
&= \sum_{\text{$d\mid n$}} \frac{n}{d} \underbrace{\sum_{J \subseteq I} (-1)^{|J|-|J/d|}\betacyc_{n/d}(J/d)}_{\Phi(d)}.
\end{align*}
We now fix $d \mid n$ and examine the expression $\Phi(d)$ (labeled above). Suppose there is some $i^* \in I$ that is not divisible by $d$. Then $J/d = (J \cup \{i^*\})/d$ for any $J$, so
\begin{align*}
\Phi(d) &= \sum_{\substack{J \subseteq I\\i^* \not\in J}} (-1)^{|J|-|J/d|}\betacyc_{n/d}(J/d) + (-1)^{|J\cup\{i^*\}| - |J/d|}\betacyc_{n/d}(J/d) \\
&= \sum_{\substack{J \subseteq I\\i^* \not\in J}} (-1)^{|J|-|J/d|}\betacyc_{n/d}(J/d) + (-1)^{|J| - |J/d| + 1}\betacyc_{n/d}(J/d) = \sum_{\substack{J \subseteq I\\i^* \not\in J}} 0 = 0.
\end{align*}
Thus, $\Phi(d) = 0$ unless $d$ divides every element of $I$, that is, $d \mid (I,n)$. We get
\begin{align*}
\alpha_n(I) &= \sum_{\text{$d \mid (I,n)$}} \frac{n}{d} \Phi(d) = \sum_{\text{$d \mid (I,n)$}} \frac{n}{d} \sum_{J \subseteq I} (-1)^{|J|-|J/d|}\betacyc_{n/d}(J/d) \\
&= \sum_{\text{$d \mid (I,n)$}} \frac{n}{d} \sum_{J \subseteq I} \betacyc_{n/d}(J/d) 
= \sum_{\text{$d \mid (I,n)$}} \frac{n}{d} \alphacyc_{n/d}(J/d). \qedhere
\end{align*}\end{proof}

Now we are ready to prove our Theorem \ref{maintheorem}. By Theorem \ref{inversions}, it suffices to prove one of (a), (b), (c), or (d). We first present a proof of part~(a).

\begin{proof}[Proof of Theorem \ref{maintheorem}(a).] Let $\mu = \co(I)$, and recall from Lemma~\ref{lem:multi} that $\alpha_n(I)$ equals the number of length-$n$ words with evaluation $\mu$.

Let $m = (I,n)$, and note that $m = \gcd(\mu_1,\mu_2,\ldots)$. In particular, if there is a word $w$ with $\ev(w) = \mu$ and $\per(w) = n/d$, then we must have $d \mid m$. We have
\begin{align*}
\alpha_n(I) = \binom{n}{\mu} &= \sum_{\text{$d \mid m$}} \#\{\text{words $w$ with $|w| = n$ and $\ev(w) = \mu$ and $\per(w) = n/d$}\} \\
&= \sum_{\text{$d \mid m$}} \#\{\text{primitive words $u$ with $|u| = n/d$ and $\ev(u) = \mu/d$}\} \\
&= \sum_{\text{$d \mid m$}} \frac{n}{d} \cdot \#\{\text{Lyndon words $u$ with $|u| = n/d$ and $\ev(u) = \mu/d$}\} \\
&= \sum_{\text{$d \mid m$}} \frac{n}{d}\, a_{n/d,\mu/d}.
\end{align*}
By Lemma \ref{sillylemma} we have $\co(I/d) = \mu/d$, so by Theorem \ref{GRcor}, $a_{n/d,\mu/d} = \alphacyc_{n/d}(I/d)$, from where the result follows.\end{proof}

Next we present an independent proof of part (b), which uses the machinery of symmetric functions introduced in Section \ref{NotationSection}.

\begin{proof}[Proof of Theorem \ref{maintheorem}(b).] Let $\mu = \co(I)$. For a symmetric function $f$, let $[x^\mu] f$ denote the coefficient of $x^\mu$ in $f$. By Theorem \ref{GRcor},
\[ \alphacyc_n(I) = [x^\mu]L_n = \frac{1}{n}\sum_{\text{$d \mid n$}} \mob(d)\, [x^\mu]p_d^{n/d}. \]
Let $m = (I,n) = \gcd(\mu_1,\mu_2,\ldots)$. Since $p_d^{n/d} = {\left(x_1^d + x_2^d + \cdots\right)}^{n/d}$, the expression $p_d^{n/d}$ is a series in $x_1^d, x_2^d, \ldots$, so $[x^\mu]p_d^{n/d}$ is zero if $m$ is not divisible by $d$. On the other hand, if $m$ is divisible by $d$, then we can make the substitution $z_i = x_i^d$, and we obtain
\[ [x^\mu]p_d^{n/d} = [z^{\mu/d}]{\left(z_1 + z_2 + \cdots\right)}^{n/d} = \binom{n/d}{\mu/d}. \]
Therefore,
\[ \alphacyc_n(I) = \frac{1}{n}\sum_{\text{$d\mid m$}} \mob(d)\,\binom{n/d}{\mu/d}. \]
By Lemma \ref{sillylemma}, we have $\co(I/d) = \mu/d$, and so $\binom{n/d}{\mu/d} = \alpha_{n/d}(I/d)$.
\end{proof}

\section{Consequences of the main theorem}\label{sec:consequences}

In this section we study some special cases of Theorem~\ref{maintheorem}, and use them to recover a few results in the literature.

%\subsection{Simple special cases}

\begin{cor} \label{cor1} Let $I \subseteq [n-1]$.
\begin{itemize}
\item[(a)] If $(I,n) = 1$, then $\alpha_n(I) = n\,\alphacyc_n(I)$. 
\item[(b)] If $\gcd(i,n) = 1$ for all $i \in I$, then $\beta_n(I) = n\,\betacyc_n(I) + (-1)^{|I|}$.
\end{itemize}
\end{cor}

\begin{proof} Part (a) follows immediately from Theorem \ref{maintheorem}(a). 

For part (b), assume $\gcd(i,n) = 1$ for all $i \in S$. Then, for $d \mid n$, we have $I/d = \varnothing$ unless $d = 1$, and $\betacyc_{n/d}(\varnothing) = 0$ unless $d = n$. Thus, by Theorem \ref{maintheorem}(d),
$$ \beta_n(I) = \sum_{\text{$d\mid n$}} (-1)^{|I|-|I/d|}\,\frac{n}{d}\,\betacyc_{n/d}(I/d) = n\,\betacyc_n(I) + (-1)^{|I|}\,\betacyc_1(\varnothing) = n\,\betacyc_n(I) + (-1)^{|I|}. \qedhere$$\end{proof}

Corollary \ref{cor1}(b) shows that, when $I$ meets the given condition, $\betacyc_n(I)$ is very close to $\smallfrac1n \beta_n(I)$. In Section \ref{sec:special}, we will see that a similar phenomenon holds for almost all sets $I$.

As a special case of Theorem \ref{maintheorem}(c), we recover the following result of Elizalde \cite{Elizalde} relating descents sets on $\Cyc_n$ and $\Perm_{n-1}$. Two proofs are presented in~\cite{Elizalde}: one is bijective, and the other uses Theorem~\ref{GRcor} and inclusion-exclusion.

\begin{cor}[{\cite[Cor.\ 4.1]{Elizalde}}] \label{CorSergi} 
For $I \subseteq [n-2]$,
\[ \#\{\pi \in \Cyc_n \colon D(\pi) \cap [n-2] = I\} = \beta_{n-1}(I). \]
Equivalently, $\betacyc_n(I) + \betacyc_n(I \cup \{n-1\}) = \beta_{n-1}(I)$.
\end{cor}

\begin{proof} Write $I^+ = I \cup \{n-1\}$. For $d \mid n$, observe that $I/d = I^+/d$ unless $d = 1$. Then, by Theorem \ref{maintheorem}(c),
\begin{align*}
&\betacyc_n(I) + \betacyc_n(I^+) \\
&= \frac{1}{n} \sum_{\text{$d \mid n$}} \mob(d)\,(-1)^{|I|-|I/d|}\,\beta_{n/d}(I/d) + \frac{1}{n} \sum_{\text{$d \mid n$}} \mob(d)\,(-1)^{|I^+|-|I^+/d|}\,\beta_{n/d}(I^+/d) \\
&= \frac{1}{n}\left[\beta_n(I) + \beta_n(I^+) + \sum_{\substack{\text{$d \mid n$}\\d \not= 1}} \mob(d)\left[(-1)^{|I|-|I/d|} + (-1)^{|I|+1-|I/d|}\right] \beta_{n/d}(I/d) \right] \\
&= \frac{1}{n}\left[\beta_n(I) + \beta_n(I^+)\right]=\beta_{n-1}(I).
\end{align*}
The last equality is proved noting that a permutation $\pi \in \Perm_n$ with descent set $I$ or $I \cup\{n-1\}$ can be obtained as follows: for the first $n-1$ entries, select a permutation $\sigma\in\Perm_{n-1}$ with descent set $I$, then select one of $n$ possible values for $\pi(n)$ to append to $\sigma$, and finally increase by one the entries of $\sigma$ greater than or equal to $\pi(n)$. \end{proof}

%A similar computation can prove \cite[Cor.\ 4.4]{Elizalde} with a little more work; the proof is omitted.

The next proposition follows easily from the generating function identity in \cite[Cor.\ 6.2]{GR}, simply by setting $q = 1$ and extracting the coefficient of $t^k$. Here we give an alternate proof using our main theorem.

\begin{prop} \label{desnum} Let $A(n,k)$ denote the Eulerian number, \textit{i.e.}\ the number of $\pi \in \Perm_n$ with exactly $k-1$ descents; and let $C(n,k)$ be the number of cycles $\pi \in \Cyc_n$ with exactly $k-1$ descents. Then
\[ C(n,k) = \frac{1}{n} \sum_{\text{$d \mid n$}} \sum_{j=0}^k \mob(d)\,(-1)^{k-j} \binom{n - \frac{n}{d}}{k-j}\,A(n/d,j). \]
\end{prop}

\begin{proof} Using Theorem \ref{maintheorem}(c),
$$ C(n,k) = \sum_{\substack{I \subseteq [n-1]\\|I| = k-1}} \betacyc_n(I) = \sum_{\substack{I \subseteq [n-1]\\|I| = k-1}}\frac{1}{n} \sum_{\text{$d\mid n$}} \mob(d)\,(-1)^{k-1-|I/d|} \beta_{n/d}(I/d).$$
Setting $J=I/d$, and summing first over $J$ and then over $I$ with $I/d = J$,
\begin{align*}
C(n,k) &= \frac{1}{n}\sum_{\text{$d\mid n$}} \sum_{J \subseteq [n/d-1]} \sum_{\substack{I \subseteq [n-1]\\|I| = k-1\\ I/d = J}} \mob(d)\,(-1)^{k-1-|J|} \beta_{n/d}(J) \\
&= \frac{1}{n}\sum_{\text{$d\mid n$}} \sum_{J \subseteq [n/d-1]} \mob(d)\,(-1)^{k-1-|J|} \beta_{n/d}(J) \cdot \#\left\{\footnotesize\begin{array}{c}I \subseteq [n-1],\\|I| = k-1,\, I/d = J\end{array}\right\} \\
&= \frac{1}{n}\sum_{\text{$d\mid n$}} \sum_{J \subseteq [n/d-1]} \mob(d)\,(-1)^{k-1-|J|} \beta_{n/d}(J) \,\binom{n-\frac{n}{d}}{k-1-|J|} \\
&= \frac{1}{n}\sum_{\text{$d\mid n$}} \sum_{j\ge1} \sum_{\substack{J \subseteq [n/d-1]\\|J| = j-1}} \mob(d)\,(-1)^{k-j} \binom{n-\frac{n}{d}}{k-j} \beta_{n/d}(J) \\
&= \frac{1}{n}\sum_{\text{$d\mid n$}} \sum_{j\ge1} \mob(d)\,(-1)^{k-j} \binom{n-\frac{n}{d}}{k-j} \sum_{\substack{J \subseteq [n/d-1]\\|J| = j-1}} \beta_{n/d}(J) \\
&= \frac{1}{n}\sum_{\text{$d\mid n$}} \sum_{j\ge1} \mob(d)\,(-1)^{k-j} \binom{n-\frac{n}{d}}{k-j}\,A(n/d,j). \qedhere
\end{align*}\end{proof}

\subsection{Complements of descent sets}

For $I \subseteq [n-1]$, let $\overline{I}$ denote the complement of $I$, namely $\overline{I} = [n-1] \smallsetminus I$. It follows from \cite[Thm.\ 4.1]{GR} that $\betacyc_n(\overline{I}) = \betacyc_n(I)$ when $n \not\equiv 2\bmod 4$. The following proposition is a full description of the case where $n \equiv 2\bmod 4$. Recall that $I/2 = \{i/2 \colon \text{$i \in I$ and $i$ is even}\}$. Note that, for $n \equiv 2\bmod 4$, exactly one of $I$ or $\overline{I}$ has an odd number of odd elements.

\begin{prop} \label{complement} If $n \equiv 2\bmod 4$ and $I \subseteq [n-1]$ has an odd number of odd elements, then $\betacyc_n(I) - \betacyc_n(\overline{I}) = \betacyc_{n/2}(I/2)$.
\end{prop}

\begin{proof}
By applying the complementation operation to permutations, $\beta_k(\overline{J}) = \beta_k(J)$. We have $\overline{I}/d = \overline{I/d}$, so $|\overline{I}/d| = n/d - 1 - |I/d|$ and $\beta_{n/d}(\overline{I}/d) = \beta_{n/d}(I/d)$. By Theorem \ref{maintheorem}(c),
\begin{align*}
\betacyc_n(I) - \betacyc_n(\overline{I}) &= \frac{1}{n} \sum_{\text{$d \mid n$}} \mob(d) \left[ (-1)^{|I| - |I/d|} \beta_{n/d}(I/d) - (-1)^{|\overline{I}| - |\overline{I}/d|} \beta_{n/d}(\overline{I}/d) \right] \\
&= \frac{1}{n} \sum_{\text{$d \mid n$}} \mob(d) \left[ (-1)^{|I| - |I/d|} - (-1)^{n - n/d - |I| + |I/d|} \right] \beta_{n/d}(I/d) \\
&= \frac{1}{n} \sum_{\text{$d \mid n$}} \mob(d)\, (-1)^{|I| - |I/d|} \left[ 1 - (-1)^{n - n/d} \right] \beta_{n/d}(I/d).
\end{align*}
Since $n \equiv 2\bmod 4$, the term $1 - (-1)^{n - n/d}$ is non-zero precisely when $d$ is even, and in that case $1 - (-1)^{n - n/d} = 2$. Thus, the sum above becomes
\begin{align*}
\frac{2}{n} \sum_{\substack{\text{$d \mid n$} \\ \text{$d$ even}}} \mob(d) \,(-1)^{|I| - |I/d|} \beta_{n/d}(I/d), \end{align*}
and we can re-index by setting $d = 2e$:
\begin{align*}
& \frac{1}{n/2} \sum_{\text{$e \mid n/2$}} \mob(2e) \,(-1)^{|I| - |I/(2e)|} \beta_{n/(2e)}(I/(2e)) \\
&= \frac{1}{n/2} \sum_{\text{$e \mid n/2$}} -\mob(e) \,(-1)^{|I| - |(I/2)/e|} \beta_{(n/2)/e}((I/2)/e). \since{by Lemma 1.1 (a)}
\end{align*}
Since $|I/2|$ is the number of even elements in $I$ and $I$ has an odd number of odd elements by assumption, $I/2$ and $I$ do not have the same parity. Thus, $-(-1)^{|I|} = (-1)^{|I/2|}$, and the sum becomes
\[ \frac{1}{n/2} \sum_{\text{$e \mid n/2$}} \mob(e) \,(-1)^{|I/2| - |(I/2)/e|} \beta_{(n/2)/e}((I/2)/e), \]
which is $\betacyc_{n/2}(I/2)$ by Theorem \ref{maintheorem}(c).
\end{proof}

The next result now follows immediately.

\begin{cor} \label{complementcor} Let $n \equiv 2 \bmod 4$, and $I \subseteq [n-1]$. We have
\begin{enumerate}[(a)]
\item $\betacyc_n(I) \ge \betacyc_n(\overline{I})$ if $I$ has an odd number of odd elements;
\item $\betacyc_n(I) = \betacyc_n(\overline{I})$ if and only if one of $I$ and $\overline{I}$ has no even elements. 
\hfill $\square$
\end{enumerate}
\end{cor}

\subsection{Cycles with descent set $\{k, 2k, 3k, \ldots\}$} \label{sec:Stanley}

In this section we prove a result of Stanley as a special case of our main theorem, then use our main theorem to obtain a generalized version of Stanley's result.

Let $E_n$ denote the $n$th Euler number, \textit{i.e.}\ $E_n = \beta_n(2\mathbb{Z} \cap [n-1])$, the number of alternating (up--down) permutations in $\Perm_n$.

\begin{cor}[{\cite[Thm.\ 5.3]{StanleyAlt}}]\label{alternating} The number of alternating cycles in $\Cyc_n$ is
$$\betacyc_n(2\mathbb{Z} \cap [n-1])=\begin{cases}
\frac{1}{n}\sum_{\text{$d \mid n$}} \mob(d)\,(-1)^{(d-1)/2}\,E_{n/d}, & \text{if $n$ is odd;}\\
\frac{1}{n}\sum_{\substack{\text{$d \mid n$}\\\text{$d$ odd}}} \mob(d)\,E_{n/d}, & \text{if $n$ is even but not a power of 2;}\\
\frac{1}{n}\left(E_n - 1\right), & \text{if $n \ge 2$ is a power of $2$.}
\end{cases}$$
\end{cor}

\begin{proof} Let $I = 2\mathbb{Z} \cap [n-1]$. In order to apply Theorem \ref{maintheorem}(c), we examine the quantities $|I/d|$ and $\beta_{n/d}(I/d)$ for $d\mid n$. If $d$ is odd, then $I/d = 2\mathbb{Z} \cap \left[\smallfrac{n}{d}-1\right]$, so $|I/d| = \left\lfloor\smallfrac{n/d-1}{2}\right\rfloor = \left\lfloor\smallfrac{n-d}{2d}\right\rfloor$, and $\beta_{n/d}(I/d) = E_{n/d}$. If $d$ is even, then $I/d = \left[\smallfrac{n}{d}-1\right]$, so $|I/d| = \smallfrac{n}{d}-1$, and $\beta_{n/d}(I/d) = 1$. 

Splitting the sum in Theorem \ref{maintheorem}(c) between odd divisors and even divisors, we obtain
\begin{align*}
\betacyc_n(I) &= \frac{1}{n} \sum_{\substack{\text{$d\mid n$}\\\text{$d$ odd}}} \mob(d)\,(-1)^{\left\lfloor \frac{n-1}{2}\right\rfloor - \left\lfloor \frac{n-d}{2d}\right\rfloor}\,E_{n/d} \;+\; \frac{1}{n}\sum_{\substack{\text{$d\mid n$}\\\text{$d$ even}}} \mob(d)\,(-1)^{\left\lfloor\frac{n-1}{2}\right\rfloor-\frac{n}{d}+1} \\
&= \frac{1}{n} \underbrace{\sum_{\substack{\text{$d\mid n$}\\\text{$d$ odd}}} \mob(d)\,\left((-1)^{n/d}\right)^{(d-1)/2}\,E_{n/d}}_{\Psi_\text{odd}} \;+\; \frac{1}{n}\underbrace{\sum_{\substack{\text{$d\mid n$}\\\text{$d$ even}}} \mob(d)\,(-1)^{\left\lfloor\frac{n-1}{2}\right\rfloor-\frac{n}{d}+1}}_{\Psi_\text{even}}.
\end{align*}

Let $\Psi_\text{odd}$ and $\Psi_\text{even}$ be as indicated above. We prove each part of the corollary as follows.
\begin{itemize}
\item Assume $n$ is odd. Then $\Psi_\text{even} = 0$. Since $n/d$ is odd, we get $(-1)^{n/d} = -1$, and so $\Psi_\text{odd} = \sum_{\substack{\text{$d\mid n$}\\\text{$d$ odd}}} \mob(d)\,(-1)^{(d-1)/2}\,E_{n/d}$.

\item Assume $n$ is even but not a power of 2, and write $n = 2^k m$ with $m > 1$ odd. Then, for odd $d \mid n$, $n/d$ is even, so $(-1)^{n/d} = 1$, and $\Psi_\text{odd} = \sum_{\substack{\text{$d\mid n$}\\\text{$d$ odd}}} \mob(d)\,E_{n/d}$. 

Now it suffices to show that $\Psi_\text{even} = 0$.
Observe that $\mob(d) = 0$ unless $d$ is square-free. If $d \mid n$ is even and square-free, then we can write $d = 2e$ with $e \mid m$, and $\mob(2e) = -\mob(e)$. Thus,
$$\Psi_\text{even} = -\sum_{\text{$e \mid m$}} \mob(e)\,(-1)^{\frac{n}{2} - \frac{n}{2e}} = -\sum_{\text{$e \mid m$}} \mob(e)\,\left((-1)^{n/e}\right)^{(e-1)/2} = -\sum_{\text{$e \mid m$}} \mob(e),$$
which is $0$ because $m > 1$.

\item Assume $n \ge 2$ is a power of 2. Then the only odd divisor of $n$ is $1$, so $\Psi_\text{odd} = E_n$. And the only even \textit{square-free} divisor of $n$ is $2$, so $\Psi_\text{even} = -1$. \qedhere
\end{itemize}\end{proof}

A very similar computation proves the analogous result for down--up permutations, \textit{i.e.}\ those with descent set $[n-1] \cap \{1, 3, 5, \ldots\}$. This result also appears in \cite[Thm.\ 5.3]{StanleyAlt}.

The above proof of Corollary~\ref{alternating} is similar to Stanley's proof of \cite[Thm.\ 5.3]{StanleyAlt}. 
His proof uses the Gessel and Reutenauer's machinery  \cite[Thm.\ 2.1]{GR} to derive a symmetric-function identity in the special case of alternating permutations \cite[Thm.\ 2.3]{StanleyAlt}, from which \cite[Thm.\ 5.3]{StanleyAlt} is then deduced.

%\[\begin{tikzpicture}
%\draw [rounded corners] (0,0) rectangle (3.5,1.5);
%\draw [rounded corners] (5.5,0) rectangle (9,1.5);
%\draw [rounded corners] (0,-3.25) rectangle (3.5,-1.75);
%\draw [rounded corners] (5.5,-3.25) rectangle (9,-1.75);
%\node (a) at (1.75,0.75) {\footnotesize$\begin{matrix}\text{Gessel and Reuten-}\\\text{auer \cite[thm.\ 2.1]{GR}}\end{matrix}$};
%\node (b) at (7.25,0.75) {\footnotesize{Stanley \cite[thm.\ 2.3]{StanleyAlt}}};
%\node (c) at (1.75,-2.5) {\footnotesize{Theorem \ref{maintheorem}}};
%\node (d) at (7.25,-2.5) {\footnotesize$\begin{matrix}\text{Corollary \ref{alternating}}\\=\\\text{Stanley \cite[thm.\ 5.3]{StanleyAlt}}\end{matrix}$};
%\node (ab) at (4.5,0.75) {\Huge$\Longrightarrow$};
%\end{tikzpicture}\]

We can extend Corollary~\ref{alternating} from alternating permutations to permutations with descent set $\{k, 2k, 3k, \ldots\}$ for any $k$. Let $E_n^{(k)}$ denote the generalized Euler numbers, \textit{i.e.}\ $E_n^{(k)} = \beta_n(k\mathbb{Z} \cap [n-1])$, which is the number of permutations with descent set $k\mathbb{Z} \cap [n-1]$.

\begin{thm} \label{generalized}
The number of cycles with descent set $k\mathbb{Z} \cap [n-1]$ is given by
\[ \betacyc_n(k\mathbb{Z} \cap [n-1]) = \frac{1}{n}\sum_{\text{$d \mid n$}} \mob(d) \, (-1)^{\left\lfloor\frac{n-1}{k}\right\rfloor - \left\lfloor\frac{n-d}{\lcm(k,d)}\right\rfloor} \, E_{n/d}^{\left(\frac{k}{\gcd(k,d)}\right)}. \]
\end{thm}

\begin{proof} Let $I = k\mathbb{Z} \cap [n-1]$. Then $|I| = \left\lfloor\smallfrac{n-1}{k}\right\rfloor$ and $I/d = \smallfrac{k}{\gcd(k,d)}\mathbb{Z}\, \cap \left[\smallfrac{n}{d}-1\right]$, so $\beta_{n/d}(I/d) = E_{n/d}^{\left(\frac{k}{\gcd(k,d)}\right)}$ and
\[ |I/d| = \left\lfloor \frac{n/d-1}{k/\gcd(k,d)}\right\rfloor = \left\lfloor \frac{n-d}{kd/\gcd(k,d)}\right\rfloor = \left\lfloor \frac{n-d}{\lcm(k,d)}\right\rfloor. \]
The desired expression follows from substituting these numbers into Theorem \ref{maintheorem}(c). \end{proof}

\begin{cor} \label{gencor1}
\renewcommand{\qedsymbol}{$\square$}\begin{proof}[\nopunct\!\!\!]If $\gcd(k,n) = 1$, then
\begin{align*} \betacyc_n(k\mathbb{Z} \cap [n-1]) &= \frac{1}{n}\sum_{\text{$d\mid n$}} \mob(d)\,(-1)^{\left\lfloor\frac{n-1}{k}\right\rfloor - \left\lfloor\frac{n-d}{k d}\right\rfloor} \, E_{n/d}^{(k)}.\qedhere\end{align*}\end{proof}
\end{cor}

It is worth considering the special case where $k$ is an odd prime.

\begin{cor} \label{gencor2}
Let $p$ be an odd prime. Then
$$\betacyc_n(p\mathbb{Z} \cap [n-1]) = \begin{cases}
\frac{1}{n}\sum_{\text{$d \mid n$}} \mob(d)\, (-1)^{\left\lfloor\frac{n-1}{p}\right\rfloor - \left\lfloor\frac{n-d}{pd}\right\rfloor}\, E_{n/d}^{(p)}, & \text{if $p \nmid n$;}\\
\frac{1}{n}\sum_{\text{$d \mid m$}} \mob(d)\,(-1)^{\frac{n(d-1)}{d}}\,E_{n/d}^{(p)}, & {\text{if $n = m p^a$ with $a\ge1$,}\atop\text{$p\nmid m$ and $m > 2$;}}\\
\frac{1}{n}\left(E_n^{(p)} + E_{n/2}^{(p)} - 2\right), & \text{if $n = 2p^a$ with $a\ge1$;} \medskip \\
\frac{1}{n}\left(E_n^{(p)} - 1\right), & \text{if $n = p^a$ with $a\ge1$.}
\end{cases}$$
\end{cor}

\begin{proof} The case where $p \nmid n$ follows immediately from Corollary \ref{gencor1}.

Now assume $p \mid n$. We will use Theorem \ref{generalized}. If $d \mid n$ and $p \nmid d$, then
\[ (-1)^{\left\lfloor\frac{n-1}{p}\right\rfloor - \left\lfloor \frac{n-d}{\lcm(p,d)}\right\rfloor} = (-1)^{\frac{n}{p} - \frac{n}{pd}} = (-1)^{n - \frac{n}{d}} = (-1)^{\frac{n(d-1)}{d}}, \]
where the second equality holds because $p$ is odd. If $d \mid n$ and $p \mid d$, then
\[ (-1)^{\left\lfloor\frac{n-1}{p}\right\rfloor - \left\lfloor \frac{n-d}{\lcm(p,d)}\right\rfloor} = (-1)^{\frac{n}{p} - \frac{n}{d}} = (-1)^{n - \frac{n}{d}} = (-1)^{\frac{n(d-1)}{d}}, \]
where again the second equality holds because $p$ is odd. Now note that $E_k^{(1)} = 1$ for all $k\ge1$. Thus, writing $n = p^a m$ with $p\nmid m$, Theorem \ref{generalized} yields
\[ \betacyc_n(p\mathbb{Z} \cap [n-1]) = \frac{1}{n} \underbrace{\sum_{\text{$d \mid m$}} \mob(d)\, (-1)^{\frac{n(d-1)}{d}}E_{n/d}^{(p)}}_{\Psi_*} \,+\, \frac{1}{n} \underbrace{\sum_{\substack{\text{$d \mid n$}\\\text{$p \mid d$}}} \mob(d)\,(-1)^{\frac{n(d-1)}{d}}}_{\Psi_0}. \]
We prove each of the remaining parts as follows.

\begin{itemize}
\item Assume $m > 2$. It suffices to show that $\Psi_0 = 0$.

Observe that $\mob(d) = 0$ unless $d$ is square-free; and if $d \mid n$ is square-free and divisible by $p$, then we can write $d = pe$ with $e \mid m$, and $\mob(pe) = -\mob(e)$. Thus,
$$\Psi_0 = -\sum_{\text{$e \mid m$}} \mob(e)\,(-1)^{\frac{n(pe-1)}{pe}} = -\sum_{\text{$e \mid m$}} \mob(e)\,(-1)^{\frac{n(e-1)}{e}}.$$ 
%\since{since $p$ is odd}
If $n$ is odd or if $4 \mid n$, then $(-1)^{\frac{n(e-1)}{e}} = 1$ for all square-free $e \mid m$, so $\Psi_0 = -\sum_{\text{$e \mid m$}} \mob(e)$, which is zero because $m > 1$.

If $n$ is even and $4\nmid n$, then $(-1)^{\frac{n(e-1)}{e}} = -(-1)^e$ for $e \mid m$. Since $m$ is even, we can write $m = 2l$; since $m$ is not divisible by $4$, $l$ is odd; and since $m > 2$, we have $l > 1$. Thus,
$$
\Psi_0 = \sum_{\text{$e \mid m$}} \mob(e)\,(-1)^e = -\sum_{\substack{\text{$e \mid m$}\\\text{$e$ odd}}} \mob(e) + \sum_{\substack{\text{$e \mid m$}\\\text{$e$ even}}} \mob(e) 
= -\sum_{\text{$e \mid l$}} \mob(e) - \sum_{\text{$f \mid l$}} \mob(f),$$
and each of these sums is zero because $l > 1$.

\item Assume $m = 2$, so $n = 2p^a$. Then $\Psi_* = E_n^{(p)} - (-1)^{p^a}E_{n/2}^{(p)} = E_n^{(p)} + E_{n/2}^{(p)}$. And $p$ and $2p$ are the only square-free divisors of $n$ that are divisible by $p$, so
\begin{align*} \Psi_0 &= -(-1)^{2p^{a-1}(p-1)} + (-1)^{p^{a-1}(2p-1)} = -2. \end{align*}

\item Assume $m = 1$, so $n = p^a$. Then $\Psi_* = E_n^{(p)}$. And $p$ is the only square-free divisor of $n$ that is divisible by $p$, so $\Psi_0 = -(-1)^{p^{a-1}(p-1)} = -1$. \qedhere
\end{itemize}\end{proof}

\section{Asymptotic results} \label{AsymptoticSection}

In this section we use the notation $f(n)\sim g(n)$ to mean $\lim_{n\to\infty} g(n)/f(n)=1$, and $f(n)\gg g(n)$ to mean $\lim_{n\to\infty} g(n)/f(n)=0$.

\subsection{Questions about asymptotic independence}\label{sec:independence}

In \cite[Sec.\ 5]{StanleySurvey}, Stanley describes the following consequence of his result expressed in Corollary \ref{alternating} above:
\begin{quote}``\ldots as $n \to \infty$, a fraction $1/n$ of the alternating permutations are $n$-cycles. Compare this with the simple fact that (exactly) $1/n$ of the permutations $w \in \Perm_n$ are $n$-cycles. We can say that the properties of being an alternating permutation and an $n$-cycle are `asymptotically independent.' What other classes of permutations are asymptotically independent from the alternating permutations?"\end{quote}

Considering that alternating permutations are those with descent set $2\ZZ\cap[n-1]$, here we ask the following related question.
\begin{question} For what other sets $I \subseteq [n-1]$ are the properties of having descent set $I$ and being an $n$-cycle asymptotically independent?
\end{question} 

To make this asymptotic question precise, one has to define the sets $I \subseteq [n-1]$ for arbitrary values of $n$. We conjecture that, in fact, as long as each $I$ is a non-empty proper subset of $[n-1]$, we asymptotically have $\frac{\betacyc_n(I)}{\beta_n(I)} \sim \frac{1}{n}$, and so the two above properties are independent in a very strong sense:

\begin{conj} \label{Conjecture} $$\lim_{n\to\infty}\max_{\varnothing \subsetneqq I \subsetneqq [n-1]} \left| \frac{n\, \betacyc_n(I)}{\beta_n(I)} - 1\right| = 0.$$
\end{conj}

We first use our main theorem to prove that an analogous relationship between $\alphacyc_n(I)$ and $\alpha_n(I)$ holds.

\begin{thm} \label{AlphaTheorem} $$\lim_{n\to\infty}\max_{\varnothing \subsetneqq I \subseteq [n-1]} \left| \frac{n\, \alphacyc_n(I)}{\alpha_n(I)} - 1\right| = 0.$$
\end{thm}

\begin{proof}
Let $I$ be a non-empty subset of $[n-1]$, let $\mu = \co(I)$, and let $m = \text{gcd}(I \cup \{n\})$. By Theorem~\ref{maintheorem}(b), 
\begin{align*}
\left| \frac{n\, \alphacyc_n(I)}{\alpha_n(I)} - 1\right| = \left| \sum_{\text{$d \mid m$}} \frac{\mob(d)\,\alpha_{n/d}(I/d)}{\alpha_n(I)} - 1\right| 
%= \left| \sum_{\substack{\text{$d \mid m$}\\d\not=1}} \frac{\mob(d)\,\alpha_{n/d}(I/d)}{\alpha_n(I)}\right| 
\le \sum_{\substack{\text{$d \mid m$}\\d\not=1}} \frac{\alpha_{n/d}(I/d)}{\alpha_n(I)} = \sum_{\substack{\text{$d \mid m$}\\d\not=1}} \frac{\binom{n/d}{\mu/d}}{ \binom{n}{\mu}}.
\end{align*}
Among the $\binom{n}{\mu}$ words with evaluation $\mu$, $\binom{n/d}{\mu/d}^d$ is the number of those of the form $w_1 \dots w_d$ where each $w_i$ is a word with evaluation $\mu/d$. Thus, $\binom{n/d}{\mu/d} \le \binom{n}{\mu}^{1/d}$, and
$$\left| \frac{n\, \alphacyc_n(I)}{\alpha_n(I)} - 1\right| \le \sum_{\substack{\text{$d \mid m$}\\d\not=1}} \binom{n}{\mu}^{1/d-1} \le \sum_{\substack{\text{$d \mid m$}\\d\not=1}} \binom{n}{\mu}^{-1/2} \le d(n)\,\binom{n}{\mu}^{-1/2}\le d(n)\,n^{-1/2},
$$
where $d(n)$ denotes the number of divisors of $n$, and we used that $\binom{n}{\mu} \ge \binom{n}{\mu_1, n-\mu_1} \ge \binom{n}{n-1, 1} = n$.
The proof is completed by the fact that $\lim_{n\to\infty} d(n)\,n^{-1/2} = 0$, which is shown in~\cite[Sec.\ 13.10]{Apostol}.
%, specifically the statement that $d(n) = o(n^\delta)$ for any $\delta > 0$). \sergi{No need to introduce the o() notation if it's only used here}
\end{proof}

In Section~\ref{sec:special} we prove more specialized versions of Conjecture~\ref{Conjecture}, using lemmas that we introduce in the remainder of this subsection. 
Just as it follows immediately from Corollary \ref{alternating} that the fraction of alternating permutations that are cycles asymptotically approaches $1/n$, it will follow (not immediately) from Theorem \ref{maintheorem}(c) that, among permutations with almost any given descent set (in a sense which we will make clear), the fraction of those that are cycles asymptotically approaches $1/n$.

%We start with a few lemmas that will be used to prove Conjecture~\ref{Conjecture} in some special cases.

\begin{lem} \label{bound2}
If $n \ge 2$ and $I \subseteq [n-1]$, then $\left|n\,\betacyc_n(I) - \beta_n(I) \right| \le (n/2)\left\lfloor n/2 \right\rfloor!$.
\end{lem}

\begin{proof}
By Theorem \ref{maintheorem}(c),
\begin{align*}
\left|n\,\betacyc_n(I) - \beta_n(I) \right| &= \left| \sum_{\substack{\text{$d \mid n$}\\\text{$d\not=1$}}} \mob(d)\,(-1)^{|I|-|I/d|}\,\beta_{n/d}(I/d) \right| \le \sum_{\substack{\text{$d \mid n$}\\\text{$d\not=1$}}} \beta_{n/d}(I/d) \\
&\le \sum_{\substack{\text{$d \mid n$}\\\text{$d\not=1$}}} \lfloor n/d \rfloor !  \le \sum_{\substack{\text{$d \mid n$}\\\text{$d\not=1$}}} \lfloor n/2 \rfloor ! \le (n/2)\,\lfloor n/2 \rfloor !. \qedhere
\end{align*}\end{proof}

The last lemma of this subsection provides the framework in which we prove the results in Section \ref{sec:special}. 

\begin{lem} \label{framework} For each $n\ge1$, let $\mathcal{I}_n \subseteq 2^{[n-1]}$; that is, $\mathcal{I}_n$ is a collection of subsets of $[n-1]$. If $\min_{I \in \mathcal{I}_n} \beta_n(I) \gg (n/2)\,\left\lfloor n/2\right\rfloor!$, then $\lim_{n\to\infty} \max_{I \in \mathcal{I}_n} \left|\frac{n\,\betacyc_n(I)}{\beta_n(I)} - 1\right| = 0$.
\end{lem}

\begin{proof} Using Lemma~\ref{bound2},
$$\max_{I \in \mathcal{I}_n} \left|\frac{n\,\betacyc_n(I)}{\beta_n(I)} - 1\right| = \max_{I \in \mathcal{I}_n} \frac{\left|n\,\betacyc_n(I) - \beta_n(I) \right|}{\beta_n(I)} 
\le \max_{I \in \mathcal{I}_n} \frac{(n/2)\,\lfloor n/2\rfloor!}{\beta_n(I)}= \frac{(n/2)\,\lfloor n/2\rfloor!}{\displaystyle\min_{I \in \mathcal{I}_n} \beta_n(I)},$$
which tends to zero as $n\to\infty$. 
\end{proof}

\subsection{Asymptotic independence in special cases} \label{sec:special}

We first prove Conjecture~\ref{Conjecture} in the case where the descent set is periodic. Let $\PP$ denote the set of positive integers. For $\ell \ge 1$, say $I \subseteq \PP$ is \emph{$\ell$-periodic} if, for each $i \in \PP$, we have $i \in I$ if and only if $i+\ell \in I$.
Our main tool is a result of Bender, Helton and Richmond:

\begin{thm}[{\cite[Thm.\ 1 \& 2]{BHR}}] \label{thm:BHR} If $I$ is $\ell$-periodic with $\varnothing \subsetneqq I \subsetneqq \PP$, and $j \ge 0$, then
\begin{align*} \beta_{\ell m}(I \cap [\ell m-1]) \sim C\,\lambda^{\ell m}\,(\ell m)! \quad\text{and}\quad
\frac{\beta_{\ell m+j}(I \cap [\ell m+j-1])}{(\ell m+j)!} \sim A_j\,\frac{\beta_{\ell m}(I \cap [\ell m-1])}{(\ell m)!}
\end{align*}
as $m \to \infty$, for some constants $C, \lambda, A_j > 0$.
\end{thm}

From Theorem \ref{thm:BHR} we easily obtain the following.
\begin{cor} \label{cor:periodic1} If $I$ is $\ell$-periodic and $\varnothing \subsetneqq I \subsetneqq \PP$, then there exist constants $C_1, C_2, \lambda > 0$ such that
\[ C_1\,\lambda^n\,n! \le \beta_n(I \cap [n-1]) \le C_2\,\lambda^n\,n! \]
for all $n$. \hfill $\square$
\end{cor}

We now use Corollary~\ref{cor:periodic1} to obtain a proof of Conjecture \ref{Conjecture} in the case where the descent set is periodic.

\begin{thm} \label{thm:periodic2} Fix $k \ge 2$, and define
\[ \mathcal{I}_n = \{ I \cap [n-1]  \colon \varnothing\subsetneqq I \subsetneqq \PP \text{ and $I$ is $\ell$-periodic for some $\ell$ with $1 \le \ell \le k$}\}. \]
Then $\lim_{n\to\infty} \max_{I \in \mathcal{I}_n} \left|\frac{n\,\betacyc_n(I)}{\beta_n(I)} - 1\right| = 0$.
\end{thm}

\begin{proof}
Let $$\mathcal{I}=\{I\colon \varnothing\subsetneqq I \subsetneqq \PP  \text{ and $I$ is $\ell$-periodic for some $\ell$ with $1 \le \ell \le k$}\},$$ and note that $\mathcal{I}_n = \{ I \cap [n-1] \colon I\in \mathcal{I}\}$. For $I \in \mathcal{I}$, Corollary \ref{cor:periodic1} says that
\[ \beta_n(I \cap [n-1]) \ge C_I \,{\lambda_I\!}^n\, n! \]
for all $n$, for some constants $C_I, \lambda_I > 0$. But there are only finitely many sets in $\mathcal{I}$, so
\[ \min_{I \in \mathcal{I}_n} \beta_n(I) = \min_{I\in \mathcal{I}} \beta_n(I \cap [n-1]) \ge C_\text{min} \,{\lambda_\text{min}\!}^n \,n!, \]
where $\lambda_\text{min}$ is the smallest $\lambda_I$ and $C_\text{min}$ is the smallest $C_I$, over all $I\in\mathcal{I}$. Since $C_\text{min} \,{\lambda_\text{min}\!}^n \,n! \gg {(n/2)\, \lfloor n/2 \rfloor!}$, the result now follows from Lemma \ref{framework}.
\end{proof}

Since the set $I=2\ZZ$ is 2-periodic, Theorem~\ref{thm:periodic2} applies in particular to alternating permutations, and so it generalizes the observation of Stanley quoted in Section~\ref{sec:independence}.

In the rest of this section we prove another special case of Conjecture~\ref{Conjecture}.

\begin{defn} \label{alternation} The \emph{alternation set} of $I \subseteq [n-1]$ is 
$$\Alt(I)=\{i\in[n-2]\colon|I\cap\{i,i+1\}|=1\},$$ 
that is, the set of $i \in [n-2]$ such that exactly one of $i$ and $i+1$ is in $I$. The \emph{alternation number} of $I$ is $\alt(I) = \card{\Alt(I)}$.
\end{defn}

Note that a given subset of $[n-2]$ is equal to $\Alt(I)$ for exactly two sets $I \subseteq [n-1]$, which are complements of each other. Note also that, for $\Alt(I)$ to be well-defined, we need to specify $n$. The value of $n$ will always be clear from context.

\begin{thm} \label{TheBigOne}
Let $\varepsilon > 0$. For each $n$, define
\[ \mathcal{I}_n = \{I \subseteq [n-1] \colon \alt(I) > n/2 - n^{1-\varepsilon}\}. \]
Then $\lim_{n\to\infty} \max_{I \in \mathcal{I}_n} \left|\frac{n\,\betacyc_n(I)}{\beta_n(I)} - 1\right| = 0$.
\end{thm}

Observe that, for a uniformly random subset of $[n-1]$, the expected alternation number is $n/2-1$. While $n/2 - n^{1-\varepsilon}$ is less than this expected value (assuming $0 < \varepsilon < 1$), it is asymptotically the same. Thus the sets in $\mathcal{I}_n$ can be thought of as sets whose alternation number is at least average or a little bit less than average.

The proof of this theorem will use Lemma \ref{framework}, so we will first establish lower bounds for $\beta_n(I)$ when $I\in\mathcal{I}_n$. The following lemma builds on the work of Ehrenborg and Mahajan \cite{EM}.

\begin{lem} \label{minimizing} Let $I \subseteq [n-1]$ and $m = \alt(I)$. Then $\beta_n(I) \ge \beta_n(\{2, 4, \ldots, m\})$ if $m$ is even, and $\beta_n(I) \ge \beta_n(\{1, 3, \ldots, m\})$ if $m$ is odd.
\end{lem}

\begin{proof}
Adapting the notation of Ehrenborg and Mahajan \cite{EM}, if $L = (l_0, \ldots, l_m)$ is a finite list of positive integers whose sum is $n-1$, then we let $\beta[L]$ denote the number of permutations in $\Perm_n$ that begin with $l_0$ ascents, next have $l_1$ descents, next have $l_2$ ascents, and so on; we can state this more compactly as $\beta[L] = \beta_n(I)$ if $L = \co(\Alt(I))$. Notice that $\beta[L]$ also equals the number of permutations that begin with $l_0$ descents, next have $l_1$ ascents, and so on. In this language, the lemma can be restated as follows: \emph{If $L$ is a list of $m+1$ positive integers whose sum is $n-1$, then $\beta[L] \ge \beta[1^m, n-m-1]$}, where we write $1^i$ to denote a sequence of $i$ ones.

By letting $a$ and $b$ be the first and second entry greater than $1$ in $L$, we can write $L = (1^i, a, 1^j, b, n_1, \ldots, n_r)$, with $a,b > 1$ and $i,j,r \ge 0$. Corollary 6.5 in \cite{EM} then states that $\beta[L] > \beta[1^i, a+b-1, 1^j, 1, n_1, \ldots, n_r]$. Equivalently, $\beta[L]$ decreases when changing the $a$ to $a+b-1$ and changing the $b$ to $1$. Note that these changes preserve both the number of entries and the sum of the entries of the list. Repeating this process yields a list with only one entry greater than $1$, and so there exists $i\ge0$ such that $\beta[L] \ge \beta[1^i, n-m-1, 1^{m-i}]$.
Finally, $\beta[1^i, n-m-1, 1^{m-i}] \ge \beta[1^m, n-m-1]$ by \cite[Thm. 5.4]{EM}.
\end{proof}

\begin{lem} \label{minimizing2}
Let $I \subseteq [n-1]$. If $\alt(I) \ge 2k$, then $\beta_n(I) \ge \beta_n(\{2, 4, \ldots, 2k\})$.
\end{lem}

\begin{proof}
Set $m = \alt(I)$; without loss of generality, assume that $m$ is even. By Lemma \ref{minimizing}, $\beta_n(I) \ge \beta_n(\{2, 4, \ldots, m\})$. Observe that $$\Alt(\{2, 4, \ldots, m\}) = \{1, \ldots, m\} \supseteq  \{1, \ldots, 2k\} =\Alt(\{2,4,\ldots,2k\}),$$
since $m \ge 2k$. We conclude that $\beta_n(\{2,4,\ldots, m\}) \ge \beta_n(\{2, 4, \ldots, 2k\})$ by using \cite[Prop.\ 1.6.4]{Stanley1}, which states that if $\Alt(J) \subseteq \Alt(J')$, then $\beta_n(J) \le \beta_n(J')$.
\end{proof}

\begin{lem} \label{EulerBound}
If $2k \le n/2$, then $\beta_n(\{2, 4, \ldots, 2k\}) \ge \frac12 \binom{n}{2k} E_{2k}$.
\end{lem}

 \begin{proof}
The lemma is trivial for $k=0$ and easily verified for $k=1$, so we assume $k\ge2$.

First we show that 
\begin{equation}\label{eq:E2i} \beta_n(\{2, 4, \ldots, 2i-2\}) + \beta_n(\{2, 4, \ldots, 2i\}) = \binom{n}{2i} E_{2i}, 
\end{equation}
for $1 \le i \le (n-1)/2$. Indeed, a permutation whose descent set is either $\{2, 4, \ldots, 2i-2\}$ or $\{2, 4, \ldots, 2i\}$ is determined by choosing $2i$ elements of $[n]$ for the first $2i$ positions (which can be done in $\binom{n}{2i}$ ways), arranging these elements into an alternating (up--down) permutation (which can be done in $E_{2i}$ ways), and finally arranging the elements in the last $n-2i$ positions in increasing order. 

From Equation~\eqref{eq:E2i}, it follows that
\begin{align*}
 \beta_n(\{2, 4, \ldots, 2k\}) &= (-1)^k\beta_n(\varnothing) + \sum_{i=1}^k (-1)^{k-i} \left[ \beta_n(\{2, 4, \ldots, 2i-2\}) + \beta_n(\{2, 4, \ldots, 2i\}) \right] \\
&%= (-1)^k + \sum_{i=1}^k (-1)^{k-i} \binom{n}{2i} E_{2i} 
= \sum_{i=0}^k (-1)^{k-i} \binom{n}{2i} E_{2i}.
\end{align*}
The terms of this alternating sum are increasing in absolute value because $2k \le n/2$, so
\begin{align*}
\beta_n(\{2, 4, \ldots, 2k\}) &\ge \binom{n}{2k} E_{2k} - \binom{n}{2k-2} E_{2k-2} \\
&= \binom{n}{2k} E_{2k} \left(1 - \frac{2k(2k-1)E_{2k-2}}{(n-2k+1)(n-2k+2)E_{2k}} \right),
\end{align*}
and now the lemma follows from the fact that 
\[ \frac{2k(2k-1)E_{2k-2}}{(n-2k+1)(n-2k+2)E_{2k}} \le \frac{E_{2k-2}}{E_{2k}}\le\frac{1}{2}.\qedhere \]
\end{proof}

 \begin{proof}[Proof of Theorem \ref{TheBigOne}.] Let $I \in \mathcal{I}_n$, meaning that $I \subseteq [n-1]$ and $\alt(I) > n/2 - n^{1-\varepsilon}$. Set $k = k(n)= \left\lfloor \frac{n/2 - n^{1-\varepsilon}}{2} \right\rfloor$. Then $\alt(I) > n/2 - n^{1-\varepsilon} \ge 2k$, so by Lemma \ref{minimizing2} we have $\beta_n(I) \ge \beta_n(\{2, 4, \ldots, 2k\})$. Next, from Lemma \ref{EulerBound}, we obtain $\beta_n(I) \ge \frac12 \binom{n}{2k} E_{2k}$. This holds for all $I \in \mathcal{I}_n$, so in fact $\min_{I \in \mathcal{I}_n} \beta_n(I) \ge \frac12 \binom{n}{2k} E_{2k}$. Using Lemma \ref{framework}, the theorem is proved once we show that
$$ \frac12 \binom{n}{2k} E_{2k} \gg  (n/2)\,\left\lfloor n/2\right\rfloor!.$$
Using that $E_{2k} \sim \frac{2\,(2k)!}{(\pi/2)^{2k+1}}$ (see \cite[Note IV.35, p.\ 269]{FS}) and Stirling's formula, we have
\begin{align*}
& \left.(n/2)\,\left\lfloor n/2\right\rfloor! \middle/ \frac12 \binom{n}{2k} E_{2k} \right. 
= \frac{(n/2)\,\left\lfloor n/2\right\rfloor!\,(2k)!\,(n-2k)!}{n!} \,\frac{2}{E_{2k}} \\
&\sim \frac{(n/2)\,\left\lfloor n/2\right\rfloor!\,(n-2k)!\,(\pi/2)^{2k+1}}{n!}
\sim \frac{(n/2)^{n/2+1}\,\sqrt{\pi(n-2k)}\,(n-2k)^{n-2k}\,(\pi/2)^{2k+1}}{n^n\,e^{n/2-2k}} \\
&\sim \frac{(n/2)^{n/2+1}\,\sqrt{\pi(n/2+n^{1-\varepsilon})}\,(n/2+n^{1-\varepsilon})^{n/2+n^{1-\varepsilon}}\,(\pi/2)^{n/2-n^{1-\varepsilon}+1}}{n^n\,e^{n^{1-\varepsilon}}} \\
&\sim \pi^{3/2}2^{-5/2}\, n^{3/2} {\left(2e^{-1}\pi^{-1}\right)}^{\!n^{1-\varepsilon}} {\left(2^{-1}\pi^{1/2}\right)}^{\!n} \,n^{n^{1-\varepsilon}} \left(1/2+n^{-\varepsilon}\right)^{n^{1-\varepsilon}} \left(1/2+n^{-\varepsilon}\right)^{n/2}.
\end{align*}
Since $1-\varepsilon < 1$, the factor of ${\left(2^{-1}\pi^{1/2}\right)}^{\!n}$ dominates as $n\to\infty$, so we conclude that the limit is $0$, noting that $2^{-1}\pi^{1/2} < 1$.
\end{proof}

Theorem~\ref{TheBigOne} shows that, when we restrict to subsets $I \subseteq [n-1]$ with $\alt(I)>n/2 - n^{1-\varepsilon}$, the properties of having descent set $I$ and being an $n$-cycle are asymptotically independent. Since $2\ZZ\cap[n-1]$ has alternation number $n-2$, this result applies to alternating permutations, thus generalizing the observation of Stanley quoted in Section~\ref{sec:independence}.

The reason we go to the trouble of using $n/2 - n^{1-\varepsilon}$ instead of simply $n/2$ is so that we can get the following result, which shows that $\frac{\betacyc_n(I)}{\beta_n(I)} \sim \frac{1}{n}$ for almost all $I \subseteq [n-1]$ in a precise sense.

\begin{cor} \label{AlmostAll} For each $n$, there is a collection $\mathcal{I}_n$ of subsets of $[n-1]$ such that $\left|\mathcal{I}_n\right|\sim 2^{n-1}$ and $\lim_{n\to\infty} \max_{I \in \mathcal{I}_n} \left|\frac{n\,\betacyc_n(I)}{\beta_n(I)} - 1\right| = 0$.
\end{cor}

\begin{proof}
Fix a real number $\varepsilon \in \left(0, \smallfrac12\right)$, and let $\mathcal{I}_n = \{I \subseteq [n-1] \colon \alt(I) > n/2 - n^{1-\varepsilon}\}$. 
By Theorem \ref{TheBigOne}, it is enough to show that $\left|\mathcal{I}_n\right|\sim 2^{n-1}$.

Note that the alternation set of a uniformly random subset of $[n-1]$ is a uniformly random subset of $[n-2]$.
Let $A_n\subseteq [n-2]$ be uniformly random. Then $|A_n|$ has a binomial distribution, and the Central-Limit Theorem says that the random variables
\[ X_n = \frac{|A_n| - n/2}{\sqrt{n}} \]
converge to a normal distribution centered at $0$, as $n \to \infty$.

In particular, $\Pr[|A_n| > n/2 - n^{1-\varepsilon}]=\Pr[X_n > -n^{1/2-\varepsilon}]$ converges to $1$ as $n \to \infty$, since $-n^{1/2-\varepsilon} \to -\infty$.
Therefore, the fraction of subsets of $[n-2]$ that have size $> n/2 - n^{1-\varepsilon}$ converges to $1$, and so the fraction of subsets $I \subseteq [n-1]$ with $\alt(I) > n/2 - n^{1-\varepsilon}$ converges to~$1$.
\end{proof}

\section{Avoiding a monotone consecutive pattern} \label{sec:monotone}

\subsection{Monotone consecutive patterns and symmetric functions}
\label{MonotoneSubsection1}

The main result of Gessel and Reutenauer \cite{GR} can also be stated in terms of symmetric functions. For a composition $\mu \vDash n$, let $h_\mu$ (resp.\ $e_\mu$) denote the homogeneous (resp.\ elementary) symmetric function associated with $\mu$. For a skew shape $\nu/\rho$, let $s_{\nu/\rho}$ denote the associated skew Schur function. Let $\langle \cdot, \cdot \rangle$ denote the usual inner product on symmetric functions.

We draw Young diagrams in English notation. For $I \subseteq [n-1]$, writing $\co(I) = (\mu_1, \ldots, \mu_k)$, let $B_I$ denote the border strip with $\mu_j$ blocks in the $j$th row from the bottom (see Figure \ref{example}).
For the sake of convenience, we will write $s_I$ instead of $s_{B_I}$ to denote the skew Schur function of degree $n$ associated with the skew shape $B_I$. Equivalently,
\[ s_I = \sum_{\substack{a_i \le a_{i+1} \text{ if } i \not\in I \\ a_i > a_{i+1} \text{ if } i \in I}} x_{a_1} \cdots x_{a_n}. \]
(This notation depends on $n$, but we suppress the $n$ for the sake of convenience.) Hence,
\begin{equation}
\left\langle h_1^n, s_I \right\rangle = [x_1\cdots x_n] s_I = \sum_{\substack{a_1 \ldots a_n \in S_n \\ a_i \le a_{i+1} \text{ if } i \not\in I \\ a_i > a_{i+1} \text{ if } i \in I}} 1
= \beta_n(I). \label{eq:sI}
\end{equation}

\begin{figure}[b]
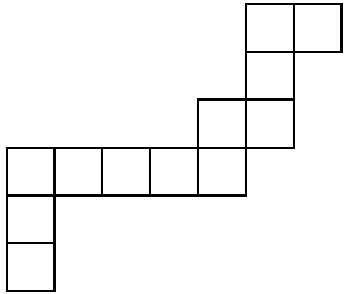

\[ \ydiagram{5+2,5+1,4+2,5,1,1} \]
\caption{\label{example}The border strip $B_I$ with $I = \{1, 2, 7, 9, 10\} \subseteq [11]$.}
\end{figure}

Recall the definition of $L_\lambda$ from Equation~\eqref{eq:L}. The following result is equivalent to Theorem \ref{GRcor}. 

\begin{thm}[{\cite[Thm.\ 2.1]{GR}}]\label{GRmain} For any $\lambda \vdash n$ and $I \subseteq [n-1]$, the number of $\pi \in \Perm_n$ with $\type(\pi) = \lambda$ and $D(\pi) = I$ is equal to $\langle L_\lambda, s_{I}\rangle$. In particular, $\betacyc(I) = \langle L_n, s_{I}\rangle$.
\end{thm}

Fix an integer $k \ge 2$. Some of our notation will depend on $k$, but we will suppress the $k$ for the sake of convenience.

We say a permutation \emph{avoids the consecutive pattern $\underline{1\dots k}$} if it does not have $k-1$ consecutive ascents. 
Consecutive patterns in permutations were introduced in~\cite{EliNoy}; see also the recent survey~\cite{Eli_survey}.
Defining
\[ G_n = \{I \subseteq [n-1] \colon \text{every part of $\co(I)$ has size $< k$}\}, \]
the number of permutations in $\Perm_n$ that avoid $\underline{1\dots k}$ is $\sum_{I \in G_n} \beta_n(I)$.

Define the symmetric function $g_n = \sum_{I \in G_n} s_{I}$.
Although we will not use this fact in the proofs below, we remark that $g_n$ is the quasi-symmetric generating function for the set of $\pi \in \Perm_n$ such that $\pi^{-1}$ avoids the consecutive pattern $\underline{1\dots k}$ (see \cite[Sec.\ 3,4]{GR} for definitions).

\begin{defn} \label{defee} For $r \ge 0$, define
\[ \ee_r = \begin{cases} 1 & \text{if $r \equiv 1 \pmod{k}$}, \\
-1 & \text{if $r \equiv 0 \pmod{k}$}, \\
0 & \text{else}. \end{cases} \]
For a composition $\mu = (\mu_1, \ldots, \mu_t)$, define $\ee_\mu = \ee_{\mu_1} \cdots \ee_{\mu_t}$.
\end{defn}

Our next theorem provides the fundamental facts about $g_n$.

\begin{thm} \label{symmetricfunction}
(a) The number of $\pi \in \Perm_n$ avoiding $\underline{1\dots k}$ is equal to $\left\langle h_1^n, g_n \right\rangle$.

(b) For $\lambda \vdash n$, the number of such $\pi$ with $\type(\pi) = \lambda$ is equal to $\left\langle L_\lambda, g_n \right\rangle$.

(c) $g_n = \sum_{\mu \vDash n} \ee_\mu h_\mu$, and
\[ \sum_{n\ge0} g_n = \left(1 - \sum_{r\ge1} \ee_r h_r\right)^{-1} = \frac{1}{1 - h_1 + h_k - h_{k+1} + h_{2k} - h_{2k+1} + \cdots}. \]
\end{thm}

\begin{proof}
Part (a) follows from Equation \eqref{eq:sI}; part (b) follows from Theorem \ref{GRmain}; part (c) is found in Gessel's Ph.D. thesis \cite[Sec.\ 5.2, Example 3]{GesselThesis}.
\end{proof}

The expression in Theorem~\ref{symmetricfunction}(a) will be useful in Section~\ref{cyc123}, but the generating function for 
permutations avoiding $\underline{1\dots k}$ was already known~\cite{DB,EliNoy}; see Equation~\eqref{eq:monotone} for the case $k=3$.

We briefly turn our attention to permutations that avoid the consecutive pattern $\underline{k\dots 1}$, that is, those that do not have $k-1$ consecutive descents. Define the symmetric function $\widetilde{g}_n = \sum_{I \in G_n} s_{\overline{I}}$, where $\overline{I} = [n-1] \smallsetminus I$. We remark that $\widetilde{g}_n$ is the quasi-symmetric generating function for the set of $\pi \in \Perm_n$ such that $\pi^{-1}$ avoids the consecutive pattern $\underline{k\dots 1}$.

Let $\omega$ denote the usual involution on symmetric functions, that is, the ring automorphism defined by $\omega(h_r) = e_r$. Since $\omega(s_{I}) = s_{\overline{I}}$, we have $\omega(g_n) = \widetilde{g}_n$, so from Theorem \ref{symmetricfunction} we obtain:

\begin{thm} \label{omega} (a) The number of $\pi \in \Perm_n$ avoiding $\underline{k\dots 1}$ is equal to $\left\langle h_1^n, \widetilde{g}_n \right\rangle$.

(b) For $\lambda \vdash n$, the number of such $\pi$ with $\type(\pi) = \lambda$ is equal to $\left\langle L_\lambda, \widetilde{g}_n \right\rangle$.

(c) $\widetilde{g}_n = \sum_{\mu \vDash n} \ee_\mu e_\mu$, and
\[ \sum_{n\ge0} \widetilde{g}_n = \left(1 - \sum_{r\ge1} \ee_r e_r\right)^{-1} = \frac{1}{1 - e_1 + e_k - e_{k+1} + e_{2k} - e_{2k+1} + \cdots}. \tag*{$\square$} \]
\end{thm}

A different proof of part (c) has been recently given by Yang and Zeilberger \cite{YZ}, who use it to obtain recurrence relations for counting various kinds of words that avoid $\underline{1 \ldots k}$.

For the remainder of this paper, we set $k = 3$, as the results do not carry over to $k > 3$.

We now consider a certain subset of the permutations avoiding $\underline{321}$, namely, the ones that begin and end with an ascent. The enumeration of these permutations is required for Theorems \ref{avoid123} and \ref{avoid321}. Define
\[ G_n^* = \{I \subseteq [n-1] \colon \text{every part of $\co(I)$ has size $\ge 2$}\}. \]
Then $\pi \in \Perm_n$ (for $n \ge 1$) has descent set in $G_n^*$ if and only if $\pi$ avoids $\underline{321}$ and $\pi$ begins and ends with an ascent; thus, the number of such permutations is $\sum_{I \in G_n^*} \beta_n(I)$.

Define the symmetric function $g_n^* = \sum_{I \in G_n^*} s_{I}$, which is the quasi-symmetric generating function for the set of $\pi \in \Perm_n$ such that $\pi^{-1}$ avoids the consecutive pattern $\underline{321}$ and $\pi^{-1}$ begins and ends with an ascent. We set the convention that $g_0^* = 1$ and $g_1^* = 0$.

\begin{defn} \label{defee2} For $r \ge 0$, define
\[ \ee_r^* = \begin{cases} 1 & \text{if $r \equiv 2 \pmod{3}$}; \\
-1 & \text{if $r \equiv 0 \pmod{3}$}; \\
0 & \text{else}. \end{cases} \]
For a composition $\mu = (\mu_1, \ldots, \mu_t)$, define $\ee_\mu^* = \ee_{\mu_1}^* \cdots \ee_{\mu_t}^*$.
\end{defn}

Our next theorem, analogous to Theorem \ref{symmetricfunction}, provides the fundamental facts about $g_n^*$.

\begin{thm} \label{symmetricfunction2} (a) The number of $\pi \in \Perm_n$ such that $\pi$ avoids $\underline{321}$ and $\pi$ begins and ends with an ascent is equal to $\left\langle h_1^n, g_n^* \right\rangle$.

(b) For $\lambda \vdash n$, the number of such $\pi$ with $\type(\pi) = \lambda$ is equal to $\left\langle L_\lambda, g_n^* \right\rangle$.

(c) $g_n^* = (-1)^n \sum_{\mu \vDash n} \ee_\mu^* h_\mu$, and
\[ \sum_{n\ge0} g_n^* = \left(1 - \sum_{r\ge1} (-1)^r \ee_r^* h_r\right)^{-1} = \frac{1}{1 - h_2 - h_3 + h_5 + h_6 - h_8 - h_9 + h_{10} + \cdots}. \]
\end{thm}
\begin{proof}
Part (a) follows from Equation (\ref{eq:sI}); part (b) follows from Theorem \ref{GRmain}; part (c) is found in Gessel's Ph.D. thesis \cite[Sec.\ 5.2, Example 4]{GesselThesis}.
\end{proof}

\subsection{Cycles avoiding $\underline{123}$ or $\underline{321}$}
\label{cyc123}

Next we find formulas counting cycles that avoid a monotone consecutive pattern of length~$3$. The analogous problem for classical patterns, that is, without the adjacency requirement, is an open problem that was proposed by Stanley in 2010, and for which there has been only limited partial progress~\cite{ArcEli,Elicont}.

We continue to set $k = 3$, and we fix the cycle type $\lambda = (n)$.
Let $\gam_n$ denote the number of permutations in $\Perm_n$ avoiding $\underline{123}$, and let $\gam_n^*$ denote the number of permutations in $\Perm_n$ avoiding $\underline{321}$ that begin and end with an ascent. We set the convention $\gam_0^* = 1$ and $\gam_1^* = 0$. In this section we will use the expressions for $\gam_n$ and $\gam_n^*$ in terms of symmetric functions given by Theorems \ref{symmetricfunction}(a) and \ref{symmetricfunction2}(a), respectively. 
However, we remark that explicit formulas for their exponential generating functions are known:
\begin{align} \label{eq:monotone} \sum_{n\ge0} \gam_n \frac{x^n}{n!}&= \left( 1 - x + \frac{x^3}{3!} - \frac{x^{4}}{4!} + \frac{x^{6}}{6!} - \frac{x^{7}}{7!} + \cdots\right)^{-1} = \frac{\sqrt{3}}{2} e^{x/2} \sec\left(\frac{\sqrt{3}}{2} x + \frac{\pi}{6}\right), \\
 \label{eq:monotone2} \sum_{n\ge0} \gam_n^* \frac{x^n}{n!}&= \left( 1 - \frac{x^2}{2!} - \frac{x^3}{3!} + \frac{x^5}{5!} + \frac{x^6}{6!} - \frac{x^8}{8!} - \cdots\right)^{-1} =  \frac{\sqrt{3}}{2} e^{-x/2} \sec\left(\frac{\sqrt{3}}{2} x + \frac{\pi}{6}\right).
\end{align}
The first expression in Equation~\eqref{eq:monotone} was obtained in 1962 by David and Barton \cite[p.\ 156--157]{DB}, and the second one by Elizalde and Noy~\cite{EliNoy} as a solution to a differential equation. Equation~\eqref{eq:monotone2} can be derived using the same methods. Note that the pattern of signs in Equations~\eqref{eq:monotone} and \eqref{eq:monotone2} is the same as in Theorems~\ref{symmetricfunction}(c) and \ref{symmetricfunction2}(c) respectively.

%Theorem \ref{symmetricfunction}(a) gives an expression for the number of permutations avoiding a monotone consecutive pattern in terms of symmetric functions, which will be useful in Section~\ref{cyc123}. The exponential generating function for these numbers was first given in 1962 by David and Barton \cite[p.\ 156--157]{DB}, as
%\[ \left( 1 - \frac{x}{1!} + \frac{x^k}{k!} - \frac{x^{k+1}}{(k+1)!} + \frac{x^{2k}}{(2k)!} - \frac{x^{2k+1}}{(2k+1)!} + \cdots\right)^{-1}. \]
%Another expression as the solution to a differential equation was later obtained in~\cite{EliNoy}.

%For the sake of completeness, we state those results here:
%\begin{thm}[\cite{EliNoy,GesselThesis}]
%\begin{align*}
%\sum_{n\ge0} \gamma_n \frac{x^n}{n!} &= {\left(\sum_{r\ge0} \ee_r \frac{x^r}{r!}\right)}^{\!-1} = {\left(1 - x + \frac{x^3}{3!} - \frac{x^4}{4!} + \frac{x^6}{6!} - \frac{x^7}{7!} + \cdots \right)}^{\!-1} \\
%&= \frac{\sqrt{3}}{2} e^{x/2} \sec\!\left(\!\frac{\sqrt{3}}{2} x + \frac{\pi}{6}\right),
%\end{align*}
%and
%\begin{align*}\sum_{n\ge0} \gamma_n^* \frac{x^n}{n!} &= {\left(\sum_{r\ge0} \ee_r^* \frac{x^r}{r!}\right)}^{\!-1} = {\left(1 - \frac{x^2}{2!} - \frac{x^3}{3!} + \frac{x^5}{5!} + \frac{x^6}{6!} - \frac{x^8}{8!} - \frac{x^9}{9!} +  \cdots \right)}^{\!-1} \\
%&= \frac{\sqrt{3}}{2} e^{-x/2} \sec\!\left(\!\frac{\sqrt{3}}{2} x + \frac{\pi}{6}\right).
%\end{align*}
%\end{thm}

Continuing with the notation from Section \ref{MonotoneSubsection1}, we use Theorem \ref{maintheorem} to find formulas for the number of cycles avoiding $\underline{123}$ (Theorem \ref{avoid123}) or avoiding $\underline{321}$ (Theorem \ref{avoid321}), in terms of the numbers $\gam_n$ and $\gam_n^*$.

\begin{thm} \label{avoid123} Define
\[ \theta(n) = \begin{cases}
1 & \text{if $n = 3^a$ with $a \ge 1$}; \\
-2 & \text{if $n = 2 \cdot 3^a$ with $a \ge 1$}; \\
0 & \text{else}.
\end{cases} \]
Then the number of $n$-cycles that avoid $\underline{123}$ is
\[ \frac{1}{n} \left[ \theta(n) + \sum_{\mathclap{\substack{\text{$d \mid n$}\\ \text{$d \equiv 1\bmod 3$}}}} \mob(d)\, \gam_{n/d} + \sum_{\mathclap{\substack{\text{$d \mid n$}\\ \text{$d \equiv 2\bmod 3$}}}} \mob(d)\, (-1)^{n/d}\, \gam_{n/d}^*\right]. \]
\end{thm}

\begin{proof}
For convenience, if $\mu = \co(I)$, we write $\beta_n(\mu) = \beta_n(I)$, and similarly for $\alpha_n(\mu)$ and $\alphacyc_n(\mu)$. By Theorem \ref{symmetricfunction}(b), setting $k = 3$, the number of $n$-cycles that avoid $\underline{123}$ is $\left\langle L_n, g_n \right\rangle$. Then, 
by Theorem \ref{symmetricfunction}(c),
\begin{align}
\left\langle L_n, g_n \right\rangle &= \sum_{\mu\vDash n} \ee_\mu \left\langle L_n, h_\mu \right\rangle= \sum_{\mu\vDash n} \ee_\mu \, [x^\mu]L_n = \sum_{\mu\vDash n} \ee_\mu\, \alphacyc_n(\mu) \since{by Theorem \ref{GRcor}} \\
&= \sum_{\mu\vDash n} \ee_\mu \,\frac{1}{n} \sum_{\text{$d \mid \gcd(\mu_1,\ldots)$}} \mob(d)\, \alpha_{n/d}(\mu/d) \since{by Theorem \ref{maintheorem}(b)} \\
&= \frac{1}{n} \sum_{\text{$d \mid n$}} \mob(d) \sum_{\substack{\mu \vDash n \\ \text{$d \mid \mu_i$}}} \ee_\mu\,\alpha_{n/d}(\mu/d)
= \frac{1}{n} \sum_{\text{$d \mid n$}} \mob(d) \sum_{\nu \vDash n/d} \ee_{d\nu}\,\alpha_{n/d}(\nu),\label{eq:star}
\end{align}
where in the last equality we set $\nu = \mu/d$, and so $\mu = d\nu = (d\nu_1, d\nu_2, \ldots)$.

Set $\varphi(d) = \sum_{\nu \vDash n/d} \ee_{d\nu}\,\alpha_{n/d}(\nu)$. To compute $\varphi(d)$, we consider three cases depending on the congruence class of $d\bmod3$.

\noindent\underline{Case $d \equiv 0 \pmod{3}$.} For compositions $\nu, \nu' \vDash n$, write $\nu \le \nu'$ if $\nu$ is a refinement of $\nu'$. Since $d\nu_i \equiv 0 \pmod{3}$ for all $i$, we have $\ee_{d\nu} = (-1)^{\ell(\nu)}$, and so
$$
\varphi(d) = \sum_{\nu \vDash n/d} (-1)^{\ell(\nu)}\,\alpha_{n/d}(\nu) = \sum_{\nu \vDash n/d} (-1)^{\ell(\nu)} \sum_{\nu'\ge\nu} \beta_{n/d}(\nu') = \sum_{\nu' \vDash n/d} \beta_{n/d}(\nu') \sum_{\nu \le \nu'} (-1)^{\ell(\nu)}.
$$
By the Principle of Inclusion--Exclusion, $\sum_{\nu \le \nu'} (-1)^{\ell(\nu)} = 0$ unless $\nu' = (1, \ldots, 1)$, hence $\varphi(d) = (-1)^{n/d}$ in this case.

\noindent\underline{Case $d \equiv 1 \pmod{3}$.} Since $d\nu_i \equiv \nu_i \pmod{3}$ for all $i$, we have $\ee_{d\nu} = \ee_\nu$, and so
\begin{align*}
\varphi(d) &= \sum_{\nu \vDash n/d} \ee_\nu\,\alpha_{n/d}(\nu) = \sum_{\nu \vDash n/d} \ee_\nu\,\binom{n/d}{\nu} = \sum_{\nu \vDash n/d} \ee_\nu\,[x^\nu]{(x_1 + x_2 + \cdots)}^{n/d} \\
&= \sum_{\nu \vDash n/d} \ee_\nu \left\langle h_1^{n/d}, h_\nu \right\rangle = \left\langle h_1^{n/d}, \sum_{\nu \vDash n/d} \ee_\nu \,h_\nu \right\rangle = \left\langle h_1^{n/d}, g_{n/d} \right\rangle=\gam_{n/d},
\end{align*}
where the last two equalities use Theorem \ref{symmetricfunction} parts (c) and (a) respectively.

\noindent\underline{Case $d \equiv 2 \pmod{3}$.} Since $d\nu_i \equiv -\nu_i \pmod{3}$ for all $i$, we have $\ee_{d\nu} = \ee_{\nu}^*$. The same calculation as the previous case shows that
\[ \varphi(d) = \left\langle h_1^{n/d}, \sum_{\nu \vDash n/d} \ee_\nu^* \,h_\nu \right\rangle = (-1)^{n/d} \left\langle h_1^{n/d}, g_{n/d}^* \right\rangle= (-1)^{n/d} \gam_{n/d}^*, \]
using Theorem \ref{symmetricfunction2}.

Combining the results of all three cases, Equation~\eqref{eq:star} becomes
\[ \left\langle L_n, g_n \right\rangle = \frac{1}{n} \left[ \hspace{.25in} \sum_{\mathclap{\substack{\text{$d \mid n$}\\ \text{$d \equiv 0\bmod 3$}}}} \mob(d)\,(-1)^{n/d} + \sum_{\mathclap{\substack{\text{$d \mid n$}\\ \text{$d \equiv 1\bmod 3$}}}} \mob(d)\, \gam_{n/d} + \sum_{\mathclap{\substack{\text{$d \mid n$}\\ \text{$d \equiv 2\bmod 3$}}}} \mob(d)\, (-1)^{n/d}\, \gam_{n/d}^*\right]. \]
It remains to prove that
\[ \sum_{\mathclap{\substack{\text{$d \mid n$}\\ \text{$d \equiv 0\bmod 3$}}}} \mob(d)\,(-1)^{n/d} = \theta(n). \]
Write $n = 3^a m$ with $3\nmid m$. If $a = 0$, then $3\nmid n$, so the sum is $0$, and $\theta(n) = 0$ too. Now assume $a \ge 1$.

If $d \mid n$ is not square-free, then $\mob(d) = 0$, so we assume $d$ is square-free. Since $3\mid d$, we can write $d = 3e$, with $e \mid m$. Then $\mob(d) = \mob(3e) = -\mob(e)$, so
\begin{equation} \sum_{\mathclap{\substack{\text{$d \mid n$}\\ \text{$d \equiv 0\bmod 3$}}}} \mob(d)\,(-1)^{n/d} = -\sum_{\text{$e \mid m$}} \mob(e)\,(-1)^{n/(3e)} = \sum_{\text{$e \mid m$}} \mob(e)\,(-1)^{(m/e)-1}, \label{eq:Mob}\end{equation}
noting that $n$ and $m$ have the same parity. If we define
\[ \theta'(m) = \begin{cases}
1 & \text{if $m = 1$}; \\
-2 & \text{if $m = 2$}; \\
0 & \text{else};
\end{cases} \]
then we see that $\sum_{\text{$e\mid m$}} \theta'(e) = (-1)^{m-1}$, so by M\"obius inversion, the right-hand side of Equation~\eqref{eq:Mob} equals $\theta'(m)$.
Therefore,
\[ \sum_{\mathclap{\substack{\text{$d \mid n$}\\ \text{$d \equiv 0\bmod 3$}}}} \mob(d)\,(-1)^{n/d} = \theta'(m) = \theta(n). \qedhere \]
\end{proof}

By symmetry, $\gamma_n$ also equals the number of permutations in $\Perm_n$ avoiding $\underline{321}$, and $\gamma_n^*$ also equals the number of permutations in $\Perm_n$ avoiding $\underline{123}$ that begin and end with a descent. We conclude this section by using the involution $\omega$ to transform Theorem \ref{avoid123} into an analogous formula for the number of $n$-cycles that avoid $\underline{321}$. 
By the remarks preceding Proposition \ref{complement}, this is equal to the number of $n$-cycles that avoid $\underline{123}$ if $n \not\equiv 2\bmod4$.

\begin{thm} \label{avoid321} Define
\[ \widetilde{\theta}(n) = \begin{cases}
1 & \text{if $n = 3^a$ with $a \ge 1$}; \\
0 & \text{else}.
\end{cases} \]
Then the number of $n$-cycles that avoid $\underline{321}$ is
\[ \frac{1}{n} \left[ \widetilde{\theta}(n) \,+\, \sum_{\mathclap{\substack{\text{$d \mid n$}\\ \text{$d \equiv 1\bmod 3$}}}} \mob(d)\,(-1)^{(d-1)n/d}\, \gam_{n/d} \,+\, (-1)^n \sum_{\mathclap{\substack{\text{$d \mid n$}\\ \text{$d \equiv 2\bmod 3$}}}} \mob(d)\, \gam_{n/d}^*\right]. \]
\end{thm}

\begin{proof}
Like in the proof of Theorem \ref{avoid123}, if $\mu = \co(I)$, we write $\beta_n(\mu) = \beta_n(I)$, and similarly for $\alpha_n(\mu)$. By Proposition \ref{omega}(b), the number of $n$-cycles that avoid $\underline{321}$ is $\left\langle L_n, \widetilde{g}_n \right\rangle$. Then
\begin{align*}
\left\langle L_n, \widetilde{g}_n \right\rangle &= \left\langle L_n, \omega(g_n) \right\rangle = \left\langle \omega(L_n), g_n \right\rangle = \sum_{\mu\vDash n} \ee_\mu \left\langle \omega(L_n), h_\mu \right\rangle \since{by Theorem \ref{symmetricfunction}(c)} \\
&= \sum_{\mu\vDash n} \ee_\mu \,\frac{1}{n} \sum_{\text{$d \mid n$}} \mob(d) \left\langle \omega(p_d)^{n/d}, h_\mu \right\rangle 
= \sum_{\mu\vDash n} \ee_\mu \,\frac{1}{n} \sum_{\text{$d \mid n$}} \mob(d)\,(-1)^{(d-1)n/d} \, [x^\mu]p_d^{n/d} \\
&= \sum_{\mu\vDash n} \ee_\mu \,\frac{1}{n} \sum_{\text{$d \mid \gcd(\mu_1,\ldots)$}} \mob(d)\,(-1)^{(d-1)n/d}\, \binom{n/d}{\mu/d} \since{see proof of Theorem \ref{maintheorem}(b)} \\
&= \sum_{\mu\vDash n} \ee_\mu \,\frac{1}{n} \sum_{\text{$d \mid \gcd(\mu_1,\ldots)$}} \mob(d)\,(-1)^{(d-1)n/d}\, \alpha_{n/d}(\mu/d) \\
&= \frac{1}{n} \sum_{\text{$d \mid n$}} \mob(d)\,(-1)^{(d-1)n/d} \sum_{\substack{\mu \vDash n \\ \text{$d \mid \mu_i$}}} \ee_\mu\,\alpha_{n/d}(\mu/d)\\
&= \frac{1}{n} \sum_{\text{$d \mid n$}} \mob(d)\,(-1)^{(d-1)n/d} \sum_{\nu \vDash n/d} \ee_{d\nu}\,\alpha_{n/d}(\nu),
\end{align*}
setting $\nu = \mu/d$, so that $\mu = d\nu = (d\nu_1, d\nu_2, \ldots)$.

The same computation as in the proof of Theorem \ref{avoid123} shows that
\[ \left\langle L_n, \widetilde{g}_n \right\rangle = \frac{1}{n} \left[ (-1)^n \sum_{\mathclap{\substack{\text{$d \mid n$}\\ \text{$d \equiv 0\bmod 3$}}}} \mob(d) + \sum_{\mathclap{\substack{\text{$d \mid n$}\\ \text{$d \equiv 1\bmod 3$}}}} \mob(d)\,(-1)^{(d-1)n/d} \,\gam_{n/d} + (-1)^n \sum_{\mathclap{\substack{\text{$d \mid n$}\\ \text{$d \equiv 2\bmod 3$}}}} \mob(d)\, \gam_{n/d}^*\right]. \]
It remains to prove that
\[ (-1)^n \sum_{\mathclap{\substack{\text{$d \mid n$}\\ \text{$d \equiv 0\bmod 3$}}}} \mob(d) = \widetilde{\theta}(n). \]
Write $n = 3^a m$ with $3\nmid m$. If $a = 0$, then $3\nmid n$, so the sum is $0$, and $\widetilde{\theta}(n) = 0$ too.

Now assume $a \ge 1$, and suppose also that $m > 1$. If $d \mid n$ is not square-free then $\mob(d) = 0$, so we assume $d$ is square-free. Since $3$ divides $d$, we can write $d = 3e$, with $e \mid m$. Then $\mob(d) = \mob(3e) = -\mob(e)$, so
\[ (-1)^n \sum_{\mathclap{\substack{\text{$d \mid n$}\\ \text{$d \equiv 0\bmod 3$}}}} \mob(d) = (-1)^{n-1} \sum_{\text{$e \mid m$}} \mob(e), \]
which is $0$ because $m > 1$, and $\widetilde{\theta}(n) = 0$ too. Finally, if $m = 1$, then $n = 3^a$, so
\[ (-1)^n \sum_{\mathclap{\substack{\text{$d \mid n$}\\ \text{$d \equiv 0\bmod 3$}}}} \mob(d) = (-1)^{3^a} \mob(3) = 1 = \widetilde{\theta}(n). \qedhere \]
\end{proof}

Unfortunately, the proof techniques for Theorems \ref{avoid123} and \ref{avoid321} do not easily generalize to the case of permutations avoiding  $\underline{1\dots k}$ or $\underline{k\dots 1}$ for $k \ge 4$.

\end{document}